\documentclass[11pt,reqno]{amsart}
\usepackage{amssymb,amsmath}

\newcommand{\threeone}{\frac{\partial t_3}{\partial t_1}}
\newcommand{\threetwo}{\frac{\partial t_3}{\partial t_2}}
\newcommand{\zeroone}{\frac{\partial t_0}{\partial t_1}}
\newcommand{\zerotwo}{\frac{\partial t_0}{\partial t_2}}
\newcommand{\Bone}{\frac{\partial B}{\partial x_1}}
\newcommand{\Btwo}{\frac{\partial B}{\partial x_2}}
\newcommand{\Bthree}{\frac{\partial B}{\partial x_3}}
\newcommand{\half}{\frac{1}{2}}
\newcommand{\R}{{\mathbb R}}
\newcommand{\T}{{\mathcal T}}
\def\sign{\operatorname{sign}}
\def\const{\operatorname{const}}
\def\Dom{\operatorname{Dom}}
\newcommand{\ci}[1]{_{{}_{\scriptstyle{#1}}}}

\newcommand{\Bel}{\mathbf B}
\newcommand{\av}[2]{\langle #1\rangle\ci {#2}}
\newcommand{\ve}{\varepsilon}

\newtheorem{theorem}{Theorem}
\newtheorem{lemma}[theorem]{Lemma}

\theoremstyle{definition}

\theoremstyle{remark}

\numberwithin{equation}{section}

\input epsf.sty
\begin{document}
\thispagestyle{empty}

\title[Monge--Amp\`ere equation for Carleson Embedding Theorems]
{{Monge--Amp\`ere equation and Bellman optimization of Carleson Embedding
Theorems}}
\author{Vasily Vasyunin}\address{Vasily Vasyunin, V. A. Steklov Math. Inst., St Petersburg
\newline{\tt vasyunin@pdmi.ras.ru}}
\author{Alexander Volberg}\address{Alexander Volberg,
Department of  Mathematics, Michigan State University and the University of Edinburgh
{\tt volberg@math.msu.edu}\,\,and\,\,{\tt a.volberg@ed.ac.uk}}

\thanks{Research of the authors was supported in part by NSF grants  DMS-0501067 (Volberg) }

\begin{abstract}
Monge--Amp\`ere equation plays an important part in Analysis. For example, it
is instrumental in mass transport problems. On the other hand, the Bellman function technique appeared recently
as a way to consider certain Harmonic Analysis problems as the problems of
Stochastic Optimal Control. This brings us to Bellman PDE, which in stochastic
setting is often a Monge--Amp\`ere equation or its close relative. We explore
the way of solving Monge--Amp\`ere equation by a sort of method of
characteristics to find the Bellman function of certain classical Harmonic
Analysis problems, and, therefore, of finding full structure of sharp constants
and extremal sequences for those problems.
\end{abstract}

\maketitle

\section{{\bf Introduction}}

The goal of this article is to show how Monge--Amp\`ere equation allows us to
solve a certain class of Harmonic Analysis problems. We choose two problems to
illustrate the method: John--Nirenberg inequality (JNI) and Carleson embeddding
theorem (CET).

In these problems we determine Bellman function of the problem (see later).
Bellman function carries all the information about the problem: sharp constant,
construction of extremal functions/extremal sequences; and sometimes it also
carries auxiliary information.

In both JNI and CET Bellman functions were found before: in \cite{SV} for JNI
and for CET in \cite{M}. However, the way they were found is rather specific
for each problem and not ``from basic principles" so-to-speak.

We propose here a universal method which fits many problems (we count at least
seven of them). Among those we chose JNI and CET bacause of their classical
nature and because the reader can easily compare our approach and our Bellman
functions with those in the literature \cite{SV}, \cite{M}. Our approach will
differ completely from the previous ones, the functions will coincide.

We borrowed the Monge--Amp\`ere solution by ``the method of characteristics"
from the literature, and the most important inspiration came from Slavin and
Stokolos  \cite{SlSt}.

Let us notice that the aim of the paper is not only to present a new and
unified approach to finding Harmonic Analysis Bellman functions. In fact, for
CET we plan to prove some new results, not just to reprove the old ones due to
\cite{M}.

\subsection{Bellman functions}
\label{bf}

The method of Bellman functions in Harmonic Analysis appeared probably in the 80's  at the series of papers of Burkholder devoted to sharp $L^p$-estimates   of martingale transform. It appeared there under certain disguise (and Bellman function and Bellman PDE were not even mentioned). The method was rediscovered in 1995 in the first preprint version of \cite{NTV1}. It was very much developed in \cite{NT}. The method turned out to be useful. Sometimes it solves certain important Harmonic Analysis problems that do not have an alternative (classical) approach so far, see e.g. \cite{PVo}.
There is now certain amount of literature on Bellman function technique. We note here an excellent paper \cite{M}, which serves as our inspiration here, and two small surveys: \cite{NTV2} and \cite{Vo}.

We fix $\ve\in (0,1)$ and an interval $I$. Subintervals of $I$ will be called
$J$. Here is what is called in \cite{SV} the Bellman function for the
John--Nirenberg inequality:

\begin{multline}
\Bel_{JNI}(x_1,x_2;\ve):=\sup\big\{\av{e^\phi}I\colon\ \av\phi I= x_1,\
\av{\phi^2}I=x_2,
\\
\av\phi J^2\le\av{\phi^2}J\le\av\phi J^2+\ve^2\quad\forall J\subset I \big\}\,.
\label{jnibel}
\end{multline}

Function $\Bel_{JNI}$ is defined in
\begin{equation}
\Omega_{JNI}:= \{x=(x_1,x_2)\colon\ x_1^2 \le x_2 \le x_1^2 +\ve^2\}\,.
\label{jnidom}
\end{equation}

Now given any $I$ we consider its dyadic lattice of subintervals $D=D(I)$ (it
consists of left and right halves (sons) of $I$ and then all left and right
grandsons and et cetera). Let $\mu$ denote a non-negative measure on $I$
without atoms. We want to write the Bellman function for the Carleson Embedding
Theorem (dyadic).

\begin{multline}
\Bel_{CET}(x_1,x_2,x_3;m,M):=\sup\frac1{|I|}\Big\{\int_I
|\phi(t)|^2\,d\mu(t)\,+\!\!\sum_{J\in D(I)} |\,\av{\phi}J|^2 \alpha_J\colon
\\
\av{\phi}I=x_1,\quad \av{\phi^2}I=x_2,\quad
\frac1{|I|}\Bigl(\mu(I)\,+\!\!\sum_{J\in D(I)}\alpha_J\Bigr)=x_3,
\\
m|J|\le\mu(J)\,+\!\!\sum_{\ell\in D(J)}\alpha_{\ell}\le M|J|\quad\forall J\in
D(I)\Big\}\,. \label{cetbel}
\end{multline}

Functions $\Bel_{CET}$ is defined in
\begin{equation}
\Omega_{CET}:= \{x=(x_1,x_2,x_3)\colon\ x_1^2 \le x_2 ,\ m\le x_3\le M\}\,.
\label{cetdom}
\end{equation}

We introduce also $\Bel(x_1,x_2,x_3)$. This is what is often called the Bellman function for the Carleson Embedding Theorem, see \cite{M}, \cite{NT}.

\begin{multline}
\label{Bel}
\Bel(x_1,x_2,x_3):=\sup\frac1{|I|}\Bigl(
\!\!\sum_{J\in D(I)} |\,\av{\phi}J|^2 \alpha_J\colon
\\
\av{\phi}I=x_1,\quad \av{\phi^2}I=x_2,\quad
\frac1{|I|}\Bigl(\!\!\sum_{J\in D(I)}\alpha_J\Bigr)=x_3,
\\
\!\!\sum_{\ell\in D(J)}\alpha_{\ell}\le |J|\quad\forall J\in
D(I)\Big\}\,.
\end{multline}

\noindent{\bf Remark.} The reader can easily see that the definitions of all
Bellman functions $\Bel_{CET},\Bel$ does not depend on the interval $I$.

The formula for the function $\Bel_{JNI}(x_1,x_2)$ was found in~\cite{V} and~\cite{SV},
the formula for $\Bel(x_1,x_2,x_3)$ was found in~\cite{M}, the general
formula for $\Bel_{CET}(x_1,x_2,x_3;m,M)$ below is new. Let us also notice that
formula for $\Bel(x_1,x_2,x_3)$ was found in~\cite{M} from intricate combinatorial consideration. We tried to explain in the present paper a natural way to find this formula from ``the basic principles".

Namely, all formulae will be
obtained here by a unified method of solving the Monge--Amp\`ere equation.

\begin{theorem}
\label{bdelta}
$$
\Bel_{JNI}(x_1,x_2;\ve) =
\frac{e^{-\ve}}{1-\ve}\left(1-\sqrt{\ve^2-(x_2-x_1^2)}\right)
e^{x_1+\sqrt{\ve^2 -(x_2-x_1^2)}}\,.
$$
\end{theorem}

\begin{theorem}
\label{bc01} Consider
\begin{equation}
\label{eqcet101} B(x_1,x_2,x_3)=
\frac{x_3x_2}{\bigl[1-2a\bigr]^2\bigl[1-4a(1-x_3)\bigr]} \,,
\end{equation}
where $a= a(x)$ is the unique solution of the cubic equation
\begin{equation}
\label{cubic01}
\frac{x_1^2}{x_2} = \left[  \frac{1-2a(1-x_3)}{1-2a} \right]^2
\frac{1-4a}{1-4a(1-x_3)}
\end{equation}
on the interval $[0,\frac1{4}]$. Then
$$
\Bel(x) =
B(x)\quad x_1^2 \le x_2,\quad 0<x_3\le 1\,.
$$
\end{theorem}

\begin{theorem}
\label{bc} Consider
\begin{equation}
\label{eqcet1} B(x_1,x_2,x_3;m,M)=
\frac{(x_3-m)x_2}{\bigl[1-2a(M-m)\bigr]^2\bigl[1-4a(M-x_3)\bigr]}+mx_2 \,,
\end{equation}
where $a= a(x)$ is the unique solution of the cubic equation
\begin{equation}
\label{cubic}
\frac{x_1^2}{x_2} = \left[  \frac{1-2a(M-x_3)}{1-2a(M-m)} \right]^2
\frac{1-4a(M-m)}{1-4a(M-x_3)}
\end{equation}
on the interval $[0,\frac1{4(M-m)}]$. Then
$$
\Bel_{CET}(x;m,M) =
\begin{cases}
B(x;m,M)\quad& x_1^2 \le x_2,\quad m<x_3\le M
\\
\qquad mx_2 &x_1^2 \le x_2,\quad x_3=m.
\end{cases}
$$
\end{theorem}

\bigskip

\noindent{\bf Remark.} The cubic equation \eqref{cubic} has actually sometimes
two solutions  and sometimes one solution (for the left hand side in $[0,1)$).
The one (called $a(x)$ above) is always in $[0,\frac1{4(M-m)}]$. Another one
(if exists) is negative. We will call it $a^-=a^-(x)$. We will see now that it
is also responsible for the meaningful and interesting extremal problem.

\bigskip

\subsection{Lower Bellman function.}
\label{minimizing}

 If we
denote the Bellman function~\eqref{cetbel} by $\Bel_{\max}$ and introduce
another Bellman function
\begin{multline}
\Bel_{\min}(x_1,x_2,x_3;m,M):=\inf\Big\{\int_I
|\phi(t)|^2\,d\mu(t)\,+\!\!\sum_{J\in D(I)} |\,\av{\phi}J|^2 \alpha_J\colon
\\
\av{\phi}I=x_1,\quad \av{\phi^2}I=x_2,\quad
\frac1{|I|}\Bigl(\mu(I)\,+\!\!\sum_{J\in D(I)}\alpha_J\Bigr)=x_3,
\\
m|J|\le\mu(J)\,+\!\!\sum_{\ell\in D(J)}\alpha_{\ell}\le M|J|\quad\forall J\in
D(I)\Big\}\,,
\label{cetbelmin}
\end{multline}
then we can show that $a=a^-$ corresponds to this extremal problem. 

\begin{theorem}
\label{bcmin} Consider
\begin{equation}
\label{eqcet1min} B(x_1,x_2,x_3;m,M)=
\frac{(x_3-m)x_2}{\bigl[1-2a(M-m)\bigr]^2\bigl[1-4a(M-x_3)\bigr]}+mx_2 \,,
\end{equation}
where $a= a^-(x)$ is the unique negative solution of the cubic equation
$$
\frac{x_1^2}{x_2} = \left[  \frac{1-2a(M-x_3)}{1-2a(M-m)} \right]^2
\frac{1-4a(M-m)}{1-4a(M-x_3)}\,.
$$
Then
$$
\Bel_{\min}(x;m,M) =
\begin{cases}
B(x;m,M)\quad& x_1^2 \le x_2,\quad \displaystyle
M-(M-m)\frac{x_1^2}{x_2}<x_3\le M \rule[-15pt]{0pt}{15pt}
\\
\qquad mx_2&x_1^2 \le x_2,\quad \displaystyle m\le x_3\le
M-(M-m)\frac{x_1^2}{x_2}\,.
\end{cases}
$$
\end{theorem}

\medskip

\noindent{\bf Discussion.}
\begin{itemize}
\item 1. We are going to provide a detailed proof of Theorem \ref{bc} only. We leave for the reader to fill out the details of the proof of Theorem \ref{bcmin}, where concavity should be repalced by convexity.
\item 2. An interesting observation follows from the
comparison of~\eqref{eqcet101} and~\eqref{eqcet1}. In fact, we can see that the
following rescaling relationship holds
\begin{equation}
\label{symm}
\Bel_{CET}(x_1, x_2, x_3; m,M)  = (M-m)\, B(x_1,x_2, \frac{x_3-m}{M-m}) + m\,x_2\,.
\end{equation}

In principle, \eqref{symm} should have followed just from the definitions of
$\Bel$, $\Bel_{CET}$ above. Let us provide a small explanation.

First, introduce the notations, by denoting $a_l$ to be the center of dyadic interval $l$, then
$$
c_l:= a_l + i \frac{|l|}{2}\,,\,\,d\alpha:= \sum_{l\in D(I)} \alpha_l\,d\delta_{c_l}\,,
$$
where $\delta_{z}$ stands for the delta measure at $z$ as usual. Actually,
\eqref{symm} means that if we are forced to have the uniform estimate from
below on how much measure we have in any {\em closed} Carleson square:
$$
(\mu + \alpha)(\overline{Q}_J) \ge m\, |J|\,,\forall J\in D(I)\,,
$$
then to maximize the quantity in the definition of $\Bel_{CET}(x;m,M)$ over all
admissible $\mu$'s and $\alpha$'s (that is to obtain $\Bel_{CET}(x;m,M)$) one needs
to keep ``boundary measure" $\mu$ to be exactly equal to $m\,dx$. In principle,
it is easy to believe that always $\mu$ should be $\ge m\,dx$. But~\eqref{symm}
claims more, namely, the equality. It means, that if we want to maximize the
outcome in the definition of $\Bel_{CET}(x;m,M)$ we need to keep minimal
possible mass (namely, $m\,dx$) on the boundary, as a reserve so-to-speak, and
we need to distribute the rest of mass in the form of measure $\alpha$
somewhere inside the half-plane. Also \eqref{symm} means that this distribution
of $\alpha$ mimics and actually equals the rescaled distribution of the
extremal measure $\alpha$ in the case when we maximize $\Bel$ and not
$\Bel_{CET}(x;m,M)$.

\item 3. Of course \eqref{symm} shows that function $\Bel$ is ``more equal'' among
$\Bel_{CET}$. But interestingly enough, \eqref{symm} seems not to follow
directly from the definitions of functions $\Bel$ and $\Bel_{CET}(x;m,M)$
in \eqref{Bel} and \eqref{cetbel} but rather from the proofs of
Theorems \ref{bc01} and \ref{bc}. In fact, it is a priori not clear why the optimal distribution of measure for $\Bel_{CET}(x;m,M)$ should have these two features: a) it leaves minimal possible ``reserve" on the boundary of the circle; b) it distributes the rest by repeating the best distribution for $\Bel$
(with correctly rescaled parameters) as \eqref{symm} shows.
\end{itemize} 

\vspace{.2in}

We shall not prove Theorem \ref{bdelta}, its proof can be found in~\cite{V}
or~\cite{SV}, in the latter paper the comparison of dyadic and non-dyadic cases
for JNI can be found as well. We use this theorem as a ``simple'' example to
illustrate how to find the Bellman function by using the general method of
solving the Monge--Amp\`ere equation. Then we use the same approach in much
more involved situations of Theorem~\ref{bc}. This Bellman function is also not
quite new, for the case $m=0$, $M=1$ the function was found by Melas
in~\cite{M}, where the reader can find the complete proof of the theorem for
this the most important case.

However, in~\cite{M} there is no explanation how the author was able to guess
his very non-trivial function. And the main accent of our paper is just the
explanation how it is possible to guess right, or more correctly, how to solve
the corresponding Bellman equation, which is a Monge--Amp\`ere equation in the
situation under consideration. Nevertheless, we will not stop after finding the
Bellman function from Theorem~\ref{bc} and provide the complete proof for the
sake of completeness and the convenience of the reader and because in such more
general setting it is new.

\section{How we proceed?}
\label{plan}

The consideration of the theorems above can be split to four parts.

\begin{itemize}
\item I. In the first part one observes that just by definition Bellman
functions $\Bel$ satisfy a certain concavity condition in their domain of
definition and boundary conditions on  (part of) the boundary of this domain.
\item II. In the second part one considers {\bf all} function satisfying these
concavity  and boundary conditions. We denote this class by $\mathcal{V}$. And
one makes the following supposition: as $b$ belongs to $\mathcal{V}$ and is the
``best'' such function, it has to satisfy not only the concavity condition but
also this concavity should become degenerate, i.e., the inequality has to turns
into equality along some vector field on our domain $\Omega$. This brings to
the picture the Monge--Amp\`ere equation. One solves it using the boundary
conditions mentioned above. The result is a function $B\in \mathcal{V}$.
Function $B$ carries an interesting geometric information to be used later.
\end{itemize}

These two steps are in fact not necessary for the proof of the results, they
are needed only to finding a function $B$, a candidate for a r\^ole of the
Bellman function $\Bel$. For example, in the excellent paper~[M] very
complicated Bellman functions appear as {\it deus ex machina}. As we shall see
the analysis of Monge--Amp\`ere equation not only supplies us with a candidate,
but it helps in the next two steps as well, namely, in the prove that the found
candidate really is the desired Bellman function.

\begin{itemize}
\item III. The third part consists of proving that $B\ge\Bel$. In convex
domains of definition this is usually not difficult. Otherwise it may require a
non-trivial proof, see \cite{SV} for non-convex $\Omega_{\ve}$.
\item IV. The
fourth part consists of proving $B\le\Bel$. This is achieved by presenting the
extremal functions or extremal sequences of functions. In its turn such
functions are found from the geometric structure of $B$ (mentioned above in
II).
\end{itemize}

\section{Monge--Amp\`ere equation.}
\label{MA}

This section is here for illustrative purposes. Through this section the
function $\Bel$ is the Bellman function $\Bel_{JNI}$ for the John--Nirenberg
inequality defined in~\eqref{jnibel} and $\Omega=\Omega_{JNI}$ is its domain
defined in~\eqref{jnidom}.

Let us very briefly review {\bf Part I} for Theorem \ref{bdelta}.

It is very easy to see that in the dyadic case (i.e. in the case when only
dyadic intervals $J\in D(I)$ are considered in~\eqref{jnibel}) the function
$\Bel$ has to satisfy concavity condition
\begin{equation}
\Bel(x) - \frac12\big(\Bel(x^+) + \Bel(x^-)\big) \ge 0\,,
\label{mi1}
\end{equation}
for all triples $x,\,x^+,\,x^-$ such that $x=\frac{x^++x^-}2$ and all three
points are from $\Omega$.

In the non-dyadic case (i.e. in the case when all subintervals $J\subset I$
are considered in~\eqref{jnibel}) it is not clear a priori why the Bellman
function $\Bel$ must satisfy the concavity condition~\eqref{mi1}. And it turns
out that for the Bellman function condition~\eqref{mi1} is not fulfilled.
However, it turns out that the Bellman function is locally concave, i.e.,
concave in any convex subdomain of $\Omega$. We shall not discuss here this
subtle moment, especially because we shall use just the local concavity
condition, which is of course weaker than the global concavity, and coincides
with the latter only for convex domains. The discussion of the difference
between dyadic and non-dyadic Bellman functions for the John--Nirenberg
inequality the reader can find in~\cite{SV}. Here we only add that for the
functions smooth enough, the concavity condition can be rewritten in the
differential form:

\begin{equation}
d^2\Bel(x) := \begin{pmatrix}
\Bel_{x_1x_1} & \Bel_{x_1x_2}\\
\Bel_{x_2x_1} & \Bel_{x_2x_2}
\end{pmatrix} \le 0\,,\qquad\forall x\in \Omega\,,
\label{mi1diff}
\end{equation}
where by $\Bel_{x_ix_j}$ we denote the partial derivatives
$\displaystyle\frac{\partial^2\Bel}{\partial x_i\partial x_j}$\,.

We have one obvious (from the definition) boundary condition on the lower parabola
of the boundary $\partial\Omega$:

\begin{equation}
\Bel(x_1, x_1^2) = e^{x_1}\,,
\label{bc1}
\end{equation}
because only the constant test functions $\phi$ correspond the points $x$ with
$x_2=x_1^2$.

We also have a simple (from definition) homogeneity condition

\begin{equation}
\Bel(x_1+t, x_2 + 2x_1 t + t^2) = e^t \Bel(x_1, x_2)\,.
\label{hm1}
\end{equation}

It follows from the definition of $\Bel$ if to take $\phi +t$ as a test
function rather than~$\phi$.

Let us move to the {\bf Part II.} Put
\begin{equation*}
\mathcal{V}=\lbrace v\in C^{2}(\Omega)\colon v \text{ satisfies }
\eqref{mi1diff},\, \eqref{bc1},\, \eqref{hm1}\rbrace\,.
\end{equation*}

Usually any function from $\mathcal{V}$ (and moreover, any function satisfying
only the concavity condition~\eqref{mi1}) supply us with some estimate. But if
we looked for a sharp estimate we need to choose the minimal possible function
$v$ from this class, which must be the Bellman function $\Bel$. This function,
``a candidate in the true Bellman function'' will be denoted by the usual
letter $B$. For every point $x\in\Omega$ there exists an extremal function
$\phi$ (or ``almost extremal'' function $\phi_n$, i.e., a sequence of
functions) realizing the supremum in the definition~\eqref{jnibel}. The usual
procedure of using the Bellman function consists in the consecutive application
of~\eqref{mi1}, when splitting the interval $I$, where a test function is
defined. For the extremal test function there has to be no lost in such
procedure, therefore, for the Bellman function the equality has to occur at
least for one splitting the point $x$ into a pair $\{x^+,\,x^-\}$. For a
concave function this means that it is linear at some direction. If we have
almost extremal functions, i.e., an extremal sequence, then we must have
``almost equality'' in~\eqref{mi1}, at least up to the terms of second order.
In any case this means that the Hessian matrix~\eqref{mi1diff} has to be
degenerated. Thus, we are looking for the ``best'' $B$, on the top of this
condition of negativity of Hessian we will impose the following {\bf
degeneration} condition:

\begin{equation}
\forall x\in \Omega\ \exists\Theta\in \R^2\setminus \{0\}\colon \
\bigl((d^2B)\,\Theta,\Theta\bigr) = \left(\begin{pmatrix} B_{x_1x_1} &
B_{x_1x_2}
\\
B_{x_2x_1} & B_{x_2x_2}
\end{pmatrix}
\begin{pmatrix} \Theta_1\\ \Theta_2\end{pmatrix},
\begin{pmatrix} \Theta_1\\ \Theta_2\end{pmatrix}\right)= 0\,.
\label{degxi1}
\end{equation}

Since the matrix $d^2 B$ is negatively defined, we conclude from~\eqref{degxi1}
the following degeneration condition on the Hessian:

\begin{equation}
\det (d^2 B) = \det \begin{pmatrix} B_{x_1x_1} & B_{x_1x_2}
\\
B_{x_2x_1} & B_{x_2x_2}
\end{pmatrix} = 0\,,\quad\forall x\in\Omega\,.
\label{degdet1}
\end{equation}

We underline once more that in the first two steps we can allow ourself not too
rigorous arguments and various assumptions, because this is not the proof, it
is heuristic way of finding a candidate in the Bellman function. The rigorous
proof that the found candidate is indeed the required Bellman function starts
from the Part III.

\noindent{\bf Claim.} {\em There is a simple algorithm to find the family of
functions enumerated by a parameter $\delta$ $(\ve\le\delta\le1)$
$$
B=B_{\delta}(x_1,x_2):=
\frac{e^{-\delta}}{1-\delta}\left(1-\sqrt{\delta^2-(x_2-x_1^2)}\right)
e^{x_1+\sqrt{\delta^2 -(x_2-x_1^2)}}
$$
that solves Monge--Amp\`ere equation~\eqref{degdet1} with boundary condition
\eqref{bc1} and homogeneity condition~\eqref{hm1}.}

To prove this claim we need

\subsection{Monge--Amp\`ere equation and method of ``characteristics"}
\label{MAE}

Let $v= v(x_1,..., x_n)$ is a smooth
function satisfying the following Monge--Amp\`ere equation in some domain
$\Omega$

\begin{equation}
\det d^2 v = \det \begin{pmatrix}
v_{x_1x_1} & \cdots & v_{x_1x_n}\\
\cdots&\cdots&\cdots\\
v_{x_nx_1} &\cdots&  v_{x_nx_n}
\end{pmatrix} = 0\,,\qquad \forall x=(x_1,\dots,x_n) \in \Omega\,,
\label{maeq}
\end{equation}
and suppose that this matrix has rank $n-1$, i.e., all smaller minors od $d^2
v$ are non-zero. Then there are functions $t_i(x_1,\dots,x_n)$, $i= 0,1,\dots,
n$, such that

\begin{equation}
v(x) = t_0 + t_1 x_1 + t_2 x_2 +\dots+ t_n x_n
\label{vrepr}
\end{equation}
and  the following $n-1$ linear equations hold:

\begin{equation}
dt_0+x_1 dt_1+x_2 dt_2+\dots+x_{n-1}dt_{n-1}+x_n dt_n=0\,.
\label{lineqrepr}
\end{equation}

Let us explain why this is $n-1$ equations and why they are linear. One needs
to read~\eqref{lineqrepr} as follows: we think that, say, $t_1,\dots,t_{n-1}$
are $n-1$ independent variables and  $t_n$, $t_0$ are functions of them.
Then~\eqref{lineqrepr} can be rewritten as follows
$$
\Bigl(\frac{\partial t_0}{\partial t_1}+x_1+ x_n\frac{\partial t_n}{\partial
t_1}\Bigr) dt_1+\dots+\Bigl(\frac{\partial t_0}{\partial t_{n-1}}+x_{n-1}+
x_n\frac{\partial t_n}{\partial t_{n-1}}\Bigr)dt_{n-1} =0\,,
$$
whence
$$
x_i+x_n\frac{\partial t_n}{\partial t_i}+\frac{\partial t_0}{\partial t_i}=0\,,
\qquad i=1,\dots,n-1\,.
$$
So we get $n-1$ equations.

\medskip

\noindent {\bf Remark.} In general we can choose {\bf any} $n-1$ variables as
independent, of course. Since the order of variables is arbitrary, sometimes
the first $n-1$ is not the most convenient choice.

\medskip

Now why these are linear equations? We think that $t_1,\dots,t_{n-1}$ is fixed.
Then the $n-1$ equations give us linear relationships in $x_1,\dots,x_n$, so
$n-1$ hyperplanes.

Therefore, \eqref{lineqrepr} gives the intersection of $n-1$ hyperplanes, so
gives us a line. We can call it $L_{t_1,\dots,t_{n-1}}$. These lines foliate
domain $\Omega$ and~\eqref{vrepr} shows that $v$ is a linear function on each
such line.

Let us prove all these propositions. Matrix $d^2 v$ annihilates one vector
$\Theta(x)$ at every $x=(x_1,...,x_n)\in\Omega$. So we get a vector field
$\Theta$. Consider its integral curve $x_1= x_1(x_n),\dots,
x_{n-1}=x_{n-1}(x_n)$. Vector $\Theta(x)$ is a tangent vector to that curve,
i.e.,
\begin{equation}
\Theta=\Theta_n
\begin{pmatrix}
x_1'\\x_2'\\ \cdots\\x_{n-1}'\\1
\end{pmatrix}.
\label{d2g}
\end{equation}
Consider a new function $g(x_n)= v(x_1(x_n),\dots,x_{n-1}(x_n), x_n)$. Due
to~\eqref{d2g} its second derivative is
\begin{equation}
 g''= \Big\langle d^2 v
\begin{pmatrix}
x_1'\\x_2'\\ \cdots\\x_{n-1}'\\1
\end{pmatrix},
\begin{pmatrix}
x_1'\\x_2'\\ \cdots\\x_{n-1}'\\1
\end{pmatrix}
\Big\rangle + v_{x_1} x_1''+\dots+ v_{x_{n-1}} x_{n-1}''= v_{x_1} x_1''+\dots+
v_{x_{n-1}} x_{n-1}''\,.
\label{t1tn}
\end{equation}
Now, let us also show that $v_{x_{1}},\dots,v_{x_{n-1}}$ are constants on this
integral curve. Suppose we are standing on the integral curve. The surface
$v_{x_1}= t_1=\const$ has normal $(v_{x_1x_1},\dots, v_{x_1x_n})$, that is the
first row of matrix $d^2 v$, which is orthogonal to the directional vector
$\Theta$ of the integral curve. Hence $\Theta$ is in the tangent hyperplane to
the surface $v_{x_1}= t_1$. The same is true for the surfaces
$v_{x_i}=t_i=\const$, $i=2,\dots, n-1$. Intersection of these surfaces gives us
our integral curves, because $\Theta$ is in the intersection of all tangent
planes to these surfaces, therefore the curves $C_{t_1,\dots,t_{n-1}}$
enumerated by constants $t_1,\dots,t_n$ are just the integral curves of the
tangent bundle $\Theta$. Thus,~\eqref{t1tn} can be rewritten as
\begin{equation}
\frac{d^2}{dx_n^2}\big(g-(t_1\,x_1+\dots+ t_n\,x_{n-1})\big)=0. \label{tttt}
\end{equation}

We obtain that the second derivative of a function (of $x_n$) in~\eqref{tttt}
is zero. So function is linear in $x_n$, that is $t_n x_n +t_0$, where the
constants $t_n$, $t_0$ depend only on the curve $C_{t_1,\dots,t_{n-1}}$, that
is
$$
t_n=t_n(t_1,\dots,t_{n-1})\,,\qquad t_0=t_0(t_1,\dots,t_{n-1})\,.
$$
We obtain on $C_{t_1,....,t_{n-1}}$
$$
v(x_1(x_n),\dots,x_{n-1}(x_n), x_n)=t_0+x_1t_1+\dots+x_{n-1}t_{n-1}+t_nx_n\,.
$$
Since we assumed our vector field to be smooth and its integral curves foliate
the whole domain, varying parameters $t_1,\dots,t_n$ we get~\eqref{vrepr}. To
check~\eqref{lineqrepr} take a full differential in~\eqref{vrepr}. Then
\begin{equation}
[v_{x_1}dx_1+\dots+v_{x_n}dx_n]=dv =[dt_0+t_1dx_1+\dots+t_ndx_n]
+x_1dt_1+\dots+x_ndt_n\,. \label{brackets}
\end{equation}
We are on $C_{t_1,\dots,t_n}$ and so $v_{x_i}=t_i$, $i=1,\dots,n-1$ as we
established already. But it is also easy to see that
$$
v_{x_n} = t_n(t_1,\dots,t_{n-1})
$$
on  $C_{t_1,\dots,t_{n-1}}$. In fact, all $v_{x_i}$ and all $t_i$ are
symmetric. We could have chosen to represent the integral curve of $\Theta$ not
as $x_i= x_i(x_n), i=1,\dots,n-1$ but as $x_j=x_j(x_1)$, $j=2,\dots,n$. Now we
see that two expression in brackets in~\eqref{brackets} are equal. Then we
obtain $x_1 dt_1+ \dots+ x_n dt_n + dt_0=0$, which is desired $n-1$ linear
relationships~\eqref{lineqrepr}.

\bigskip

\subsection{The proof of the claim}
\label{claimproof}

We can use the previous Section~\ref{MAE}. But we prefer to repeat it for the
simple case of only two variables!

\begin{proof}

Equation~\eqref{degdet1} of course is solved by linear function in $x_1$,
$x_2$, but linear function cannot satisfy the boundary condition. So  we assume
that matrix $d^2B$ is negatively defined, non-zero, and there exists exactly
one vector $\Theta\in\R^2\setminus \{0\}$ that annihilates it: $d^2B\Theta=0$.
This vector may depend of course on $(x_1,x_2)$ and we get a vector field
$\Theta(x)$ in $\Omega$. Let us see now that $B_{x_i}, i=1,2,$ are  constants on the
integral curves of this vector field. Locally our integral curves are level
sets of a certain function $s=s(x_1,x_2)$. Of course, the function $s(x)$ is
not uniquely defined, it is defined up to change of variables
$s\mapsto\varphi\circ s$, where $\varphi$ is a function of one variable. But in
any parametrization the vector $(-s_{x_2}, s_{x_1})$ is a tangent vector to the
integral curve, i.e., it is collinear with our vector field $\Theta$ and thus
annihilates $d^2B$:
\begin{equation}
\begin{pmatrix}
B_{x_1x_1} & B_{x_1x_2}\\
B_{x_2x_1} & B_{x_2x_2}
\end{pmatrix} \begin{pmatrix} -s_{x_2}\\ s_{x_1}\end{pmatrix}=0\,.
\label{theta}
\end{equation}
Notice that both functions $B_{x_i}(x_1,x_2)$ instead of $s(x_1,x_2)$ satisfy
the last equation as well. In fact, the vector $(-(B_{x_1})_{x_2},
(B_{x_1})_{x_1})$ annihilates the first row of $d^2B$ obviously. Then it
annihilates the second row just by~\eqref{degdet1}. The same is true for the
vector $(-(B_{x_1})_{x_2}, (B_{x_1})_{x_1})$. Therefore, the integral curves of
our vector field are the level sets of the functions both $B_{x_1}(x_1,x_2)$
and $B_{x_2}(x_1,x_2)$. Thus, we can conclude that (at least locally)
$B_{x_i}=t_i\circ s$ for some functions $t_i$ of one variable. Moreover, this
means that we can take any of $B_{x_i}$ to play the role of $s$: to define the
integral curves of our vector field $\Theta$ as its level sets.

Let us check now that the function $t_0:=B-t_1x_1-t_2x_2$ is also constant
along any integral curve. For this aim we have to check that we can put $t_0$
instead of $s$ into~\eqref{theta}, but it is evident, because at each point
$(x_1,x_2)$ we have
\begin{align*}
-\frac{\partial t_0}{\partial x_2}=&-\frac{\partial B}{\partial x_2}+
\frac{\partial t_1}{\partial x_2}x_1+\frac{\partial t_2}{\partial x_2}x_2+t_2=
B_{x_1x_2}x_1+B_{x_2x_2}x_2
\\
\frac{\partial t_0}{\partial x_1}=&\frac{\partial B}{\partial x_1} -
\frac{\partial t_1}{\partial x_1}x_1-\frac{\partial t_2}{\partial x_1}x_1-t_1=
-B_{x_1x_1}x_1-B_{x_2x_1}x_2
\end{align*}

So, we obtained
\begin{equation}
B(x)=t_0+t_1x_1+t_2x_2.
\label{v2}
\end{equation}
The integral curve of the kernel field of the Hessian of this function are the
straight lines given by the equation
\begin{equation}
dt_0+x_1dt_1+x_2dt_2=0\,.
\label{dv2}
\end{equation}
Indeed, on the one hand
$$
dB=B_{x_1}dx_1+B_{x_2}dx_2\,,
$$
on the other hand the differential of~\eqref{v2} is
$$
dB=dt_0+x_1dt_1+t_1dx_1+x_2dt_2+t_2dx_2\,.
$$
Comparison of the these two equalities yields~\eqref{dv2}.

If the vertical integral curves of our vector field are nowhere dense (so they
could be given by $x_2=x_2(x_1)$), then they could be considered as the level
sets $t_2=t$, and two other parameters are functions of $t$.
Equations~\eqref{v2} and~\eqref{dv2} could be rewritten as
\begin{equation}
B(x)=t_0(t)+t_1(t)x_1+tx_2\,,\qquad t_0(t)'+t_1(t)'x_1+x_2=0\,.
\end{equation}
If the horizontal integral curves of our vector field are nowhere dense, they
could be written by $x_1=x_1(x_2)$, or as the level sets $t_1=t$, and two other
parameters are functions of $t$. Equations~\eqref{v2} and~\eqref{dv2} then
takes the form
\begin{equation}
B(x)=t_0(t)+tx_1+t_2(t)x_2\,,\qquad t_0(t)'+x_1+t_2(t)'x_2=0\,.
\end{equation}

Till now we have considered general equation when the determinant of the
Hessian for a function of two variables is zero. Now we return to our specific
problem and use~\eqref{hm1} to determine the structure of the family of
integral lines.

Differentiating \eqref{hm1} in $t$ and setting $t=0$ we get
$$
B_{x_1} + 2x_1 B_{x_2} =B\,
$$
or
$$
t_1+2x_1t_2=t_0+t_1x_1+t_2x_2\,.
$$
This is the equation of a straight line
\begin{equation}
(t_0-t_1)+x_1(t_1-2t_2)+x_2t_2=0
\label{extremal}
\end{equation}
representing a level set of the
functions $t_i$, therefore it has to be the same straight line that is given by
equation~\eqref{dv2}. Hence the following two equations must be true
$$
\frac{dt_0}{t_0-t_1}=\frac{dt_1}{t_1-2t_2}=\frac{dt_2}{t_2}\,.
$$
They could be rewritten as
\begin{align}
d\Big(\frac{t_0}{t_2}\Big)=-t_1\frac{dt_2}{t_2^2}\,,
\label{eqt0}
\\
d\Big(\frac{t_1}{t_2}\Big)=-2\frac{dt_2}{t_2}\,.
\label{eqt1}
\end{align}
Solution of~\eqref{eqt1} is
\begin{equation}
t_1=-2t_2\log|t_2|+2c_1t_2 \label{t1}
\end{equation}
(we denoted the integration constant by $2c_1$ for the future convenience). After
plugging  this solution into~\eqref{eqt0}, we get
\begin{equation}
t_0=t_2\log^2|t_2|-2c_1t_2\log|t_2|+c_2t_2\,. \label{t0}
\end{equation}

Now we rewrite equation~\eqref{extremal} using the obtained expressions for
$t_i$:
$$
\Bigl(\log^2|t_2|+2(1-c_1)\log|t_2|+(c_2-2c_1)\Bigr)-2x_1\Bigl(\log|t_2|-c_1+1\Bigr)+x_2=0\,.
$$

It is convenient introduce a new variable
$$
a:=\log|t_2|-c_1+1\,.
$$
Then the formulas~\eqref{t1} and~\eqref{t0} can be rewritten as follows:
\begin{align}
t_1=&2(1-a)t_2 \label{t1a}\,,
\\
t_0=&\big((1-a)^2-c_1^2+c_2\big)t_2 \label{t0a}\,,
\end{align}
and the extremal lines are
\begin{equation}
\label{alines} x_2-2ax_1 + a^2+c_2-1-c_1^2=0\,.
\end{equation}
Since
$$
c_1^2+1-c_2=x_2-2ax_1 + a^2=(x_2-x_1^2)+(x_1-a)^2\ge0\,,
$$
we can introduce a new non-negative constant $\delta$ such that
$$
\delta^2=c_1^2+1-c_2
$$
and rewrite equation~\eqref{alines} in the form
\begin{equation}
\label{tanlines} x_2-2ax_1 + a^2-\delta^2=0\,.
\end{equation}

We observe that these are lines tangent to parabola $x_2 = x_1^2 +\delta^2$ at
the point $x_1=a$, $x_2=a^2+\delta^2$ and intersecting the lower boundary
$x_2=x_1^2$ at the points with $x_1=a\pm\delta$. We want these lines to foliate
domain $\Omega=\{(x_1, x_2): x_1^2 \le x_2 \le x_1^2 + \ve^2\}$. Hence,
$$
\delta\ge\ve\,.
$$

Now we can almost write down $B(x_1,x_2)$. The variable $a$ can be found in
terms of $x_i$ by solving equation~\eqref{tanlines}:
$$
a=x_1+\sqrt{\delta^2 -(x_2-x_1^2)}\,.
$$
We have chosen the solution with the plus sign in front of the square root. The
explanation, why  does it correspond to our problem, and to what extremal
problem corresponds the opposite sign,  can be found  in~\cite{V, SV}.  Using
now~\eqref{t1a}, \eqref{t0a}, and~\eqref{tanlines} we obtain the expression for
$B$:
\begin{align*}
B(x_1,x_2)&=t_0+t_1x_1+t_2x_2
\\
&=\big((1-a)^2-c_1^2+c_2+2(1-a)x_1+x_2\big)t_2
\\
&=2(1-a+x_1)t_2=\pm2(1-a+x_1)e^{a+c_1-1}
\\
&=\pm2\Big(1-\sqrt{\delta^2 -(x_2-x_1^2)}\Big)e^{x_1+\sqrt{\delta^2
-(x_2-x_1^2)}+c_1-1}\,.
\end{align*}
Since $B$ has to be positive and to satisfy boundary condition~\eqref{bc1} we
choose the sign and the value of the constant $c_1$. Finally, we get a family
of solutions
$$
B(x_1,x_2;\delta)=\frac{1-\sqrt{\delta^2
-(x_2-x_1^2)}}{1-\delta}e^{x_1+\sqrt{\delta^2 -(x_2-x_1^2)}-\delta}
$$
depending on a parameter $\delta$. Acceptable values of the parameter $\delta$
are $0<\delta<1$. Why it is so, as well as the choice of this parameter for a
given $\epsilon$ (by the way, different for the dyadic and non-dyadic cases),
an explanation of all these facts can be found in~\cite{SV}.
\end{proof}

The above family of functions $B$ was found in~\cite{V} and~\cite{SV} from
different reasoning. After finding this candidate the steps I, II are
completely finished for John--Nirenberg inequality. The reader is referred to
\cite{SV}, where steps III and IV are done. The fact that $\Omega$ is not convex makes
step III rather delicate.

\medskip

Now we are going to present for the reader all four steps for another more
difficult problem of finding the Bellman functions of Carleson embedding
theorems (CET).

\section{Step I for Carleson embedding theorems }
\label{ICET}

From the definition of function $\Bel$ above we can immediately see that it
satisfies the inequality (just concavity)

\begin{equation}
\label{miB} \Bel(x) - \half \big(\Bel(x^-) + \Bel(x^+)\big) \ge
0\,,\quad\forall x^\pm\in\Omega\,,\quad x= \half (x^+ + x^-)\,.
\end{equation}
Note that now $\Omega=\Omega_{CET}$ (see~\eqref{cetdom}) is convex and for all
$x^\pm\in\Omega$ we have $x\in\Omega$.

For the sake of brevity, we shall usually omit the parameters $m$ and $M$.

The boundary conditions also follow  from the definition:

\begin{equation}
\label{BCB} \Bel(x_1, x_1^2, x_3) = x_1^2\, x_3\,.
\end{equation}

\begin{equation}
\label{BCBUL} \frac{\partial\Bel}{\partial x_3} (x_1,x_2,M;m,M) = x_1^2\,.
\end{equation}

It is equally easy to see that homogeneity condition holds

\begin{equation}
\label{hmB} \Bel(tx_1, t^2x_2, x_3) = t^2\,\Bel(x_1,x_2,x_3)\,,\,\, t\in\R\,.
\end{equation}
The detailed explanation of these properties can be found in~\cite{NT}, where
this Bellman function was defined and a majorant (so called supersolution) was
found.

Let us consider all smooth functions in $\Omega$ satisfying~\eqref{miB},
\eqref{BCB}, \eqref{hmB}. Call this family $\mathcal{V}$.

As before, we are looking for the ``best" $B\in \mathcal{V}$, so on the top of
the condition~\eqref{miB} of negativity of Hessian we will impose the following
{\bf degeneration}  condition:

\begin{multline}
\label{degxiB} \forall x=(x_1,x_2, x_3) \in \Omega\quad \exists \Theta\in
\R^3\setminus \{0\}
\\
\big((d^2B)\Theta,\Theta\big) = \left(\begin{pmatrix}
B_{x_1x_1} & B_{x_1x_2} & B_{x_1x_3}\\
B_{x_2x_1} & B_{x_2x_2} & B_{x_2x_3}\\
B_{x_3x_1} & B_{x_3x_2} & B_{x_3x_3}
\end{pmatrix} \begin{pmatrix} \Theta_1\\ \Theta_2\\ \Theta_3\end{pmatrix} ,
\begin{pmatrix} \Theta_1\\\Theta_2\\ \Theta_3\end{pmatrix}\right)= 0\,.
\end{multline}

As matrix $d^2B$ is negatively defined we conclude from \eqref{degxi1} the
following degeneration condition on the Hessian:

\begin{equation}
\label{degdetB}
\det (d^2 B) = \det \begin{pmatrix}
B_{x_1x_1} & B_{x_1x_2} & B_{x_1x_3}\\
B_{x_2x_1} & B_{x_2x_2} & B_{x_2x_3}\\
B_{x_3x_1} & B_{x_3x_2} & B_{x_3x_3}\end{pmatrix} = 0\,,\qquad \forall x\in
\Omega\,.
\end{equation}

\bigskip

\noindent{\bf Claim.} {\em There is a simple algorithm to find the function
$B(x;m,M)$  that solves Monge--Amp\`ere equation~\eqref{degdetB} in the
domain~\eqref{cetdom} with boundary conditions~\eqref{BCB}--\eqref{BCBUL} and
homogeneity condition~\eqref{hmB}.}

\bigskip

\begin{proof}
Again we consider the vector field $\Theta$ such that $d^2B(x)\Theta(x)=0$,
$x=(x_1,x_2, x_3) \in \Omega$. And we consider its integral curves. Section
\ref{MAE} shows that these integral curves are straight lines (segments of
straight lines). Section \ref{MAE} and its the method of characteristics allows
us to write the following ``parametric'' equation for these lines, which we
first write in the invariant form (compare with \eqref{dv2})

\begin{equation}
\label{invlinesB} x_1 dt_1 + x_2 dt_2 + x_3 dt_3 + dt_0=0\,.
\end{equation}

This can be rewritten as follows if we choose $(t_1, t_2)$ as the set of
independent parameters defining lines $L_{t_1, t_2}$ foliating our domain:

\begin{equation}
\label{linesB}
\begin{cases}
x_1 + \threeone\cdot x_3 + \zeroone = 0
\\
x_2 + \threetwo\cdot x_3 + \zerotwo =0
\end{cases}
\end{equation}
And similarly to the previous section the solution of the Monge--Amp\`ere
equation is given by
\begin{equation}
\label{B} B(x_1, x_2, x_3) = t_1\cdot x_1 + t_2\cdot x_2 + t_3\cdot x_3 +
t_0\,.
\end{equation}
Here each line is given by fixing two free parameters $(t_1, t_2)$, which are
$t_1= \Bone$, $t_2= \Btwo$. (This can be obtained exactly as in the previous
section.) Parameters $t_3, t_0$ are not free, they are unknown functions of
$t_1$, $t_2$, for which we will find (a priori non-linear) PDE. They will be
easy in our case (and linear). And we will solve them easily. This will bring
us the formula for $B$ in the same way we get when proving the claim of
Section~\ref{MA}.

Let us use \eqref{hmB} now. Differentiating in $t$ and setting $t=1$ we get
\begin{equation}
x_1 t_1 + 2 x_2 t_2 = 2 B\,.
\label{shortB}
\end{equation}
Whence,
\begin{equation}
\label{planeB} t_1 x_1 + 2 t_3 x_3 + 2 t_0 =0\,.
\end{equation}

Homogeneity gave us \eqref{planeB} and we conclude that
this plane (for each fixed $t_1$, $t_2$) contains the line \eqref{linesB}.  But
the line \eqref{linesB} is passing through the point
\begin{equation}
\label{pointon}
( -\zeroone, -\zerotwo, 0)\,,
\end{equation}
and has the direction
\begin{equation}
\label{direction}
(-\threeone, -\threetwo, 1)\,.
\end{equation}
Hence we get from \eqref{planeB}:



\begin{equation}
\begin{cases}
\displaystyle 2 t_3-t_1\,\threeone=0\\
\displaystyle 2 t_0-t_1\,\zeroone=0\,.
\end{cases}
\end{equation}

From these ``PDE" we easily write down
\begin{equation}
\label{AD}
\begin{cases}
t_3 = A(t_2) t_1^2\\
t_0 = D(t_2) t_1^2\,.
\end{cases}
\end{equation}
Then the equations of the extremal lines $L_{t_1,t_2}$ can be rewritten in the
form
\begin{equation}
\label{linesAD}
\begin{cases}
x_1 + 2 t_1 A x_3 + 2 t_1 D=0
\\
x_2 + t_1^2 A' x_3 + t_1^2 D' =0
\end{cases}
\end{equation}

We need to work a bit to define functions $A$ and $D$. We assume that
our line intersects $\partial\Omega$ in a point
$\zeta=\zeta(t_1,t_2)$ on the ``upper lid'' $x_3=M$ and in a point
$\xi=\xi(t_1,t_2)$ on the ``side'' $x_2=x_1^2$. Then we have two pairs of
equations~\eqref{linesAD} asserting that our points are on the line
$L_{t_1,t_2}$
\begin{align}
\zeta_1 + 2 t_1 A M + 2 t_1 D&=0\label{zeta1}
\\
\zeta_2 + t_1^2 A' M + t_1^2 D'&=0\label{zeta2}
\\
\xi_1 + 2 t_1 A \xi_3 + 2 t_1 D&=0\label{xi11}
\\
\xi_1^2 + t_1^2 A' \xi_3 + t_1^2 D'&=0\label{xi12}
\end{align}
and two boundary conditions~\eqref{BCB} and~\eqref{BCBUL}
\begin{align}
\half\xi_1 t_1 + \xi_1^2 t_2&=\xi_1^2\xi_3\label{xi13}
\\
A t_1^2&=\zeta_1^2\,.\label{A}
\end{align}

Thus, we have six equations with six unknown functions: $\zeta_1$, $\zeta_2$,
$\xi_1$, $\xi_3$, $A$, and $D$. Equation~\eqref{zeta2} determines the function
$\zeta_2$. If we take $\zeta_1$ from~\eqref{zeta1} and plug into~\eqref{A} we
get
\begin{equation}
A=4(AM+D)^2\,.\label{A1}
\end{equation}
So, introducing a new function of $t_2$
\begin{equation}
a=2(AM+D)\label{a}
\end{equation}
we can express both functions $A$ and $D$ in terms of $a$: \eqref{A1} yields
\begin{equation}
A=a^2\label{A2}
\end{equation}
and directly from the definition of $a$~\eqref{a} we get
\begin{equation}
D=\half a-Ma^2\,.\label{D}
\end{equation}

Note that deducing~\eqref{A1} we divided both parts of~\eqref{A} over $t_1$.
This is a correct operation because the Bellman function $\Bel$ clearly depends
on all variables $x_i$, and therefore its partial derivatives $t_i$ cannot be
identically zero. By the way, since $t_3$ is not identically zero, so is $a$.
Moreover, we can assert that $a$ is not a constant function. Indeed, assuming
$a$ to be a constant, we have $\zeta_2=0$ from~\eqref{zeta2}, whence
$\zeta_1=0$ due to $\zeta_1^2\le\zeta_2=0$. But~\eqref{zeta1} can be written as
$\zeta_1+at_1=0$, and since neither $t_1$ nor $a$ is not zero, we come to a
contradiction.

In result we reduce our system to the system of three equations with three
unknown functions $a(t_2)$, $\xi_1(t_1,t_2)$, and $\xi_3(t_1,t_2)$. Indeed,
equations~\eqref{xi11}, \eqref{xi12}, and~\eqref{xi13} can be rewritten as
follows
\begin{align}
\xi_1 + t_1 a \big(1-2a(M-\xi_3)\big)&=0\label{xi21}
\\
\xi_1^2 + t_1^2 a'\big(\half-2a(M-\xi_3)\big)&=0\label{xi22}
\\
\half\xi_1 t_1 -\xi_1^2(\xi_3-t_2)&=0\,.\label{xi23}
\end{align}

To have possibility divide~\eqref{xi23} over $\xi_1$, we need to check that
$\xi_1=0$ cannot be an appropriate solution. Assuming that identically $\xi_1=0$ we get
$1-2a(M-\xi_3)=0$ from~\eqref{xi21}, which turns~\eqref{xi22} into $-\half
t_1^2a'=0$. But this is impossible because, $t_1\ne0$ and $a$ is not a constant
function. So, we can rewrite~\eqref{xi23} as
\begin{equation}
\xi_3=t_2+\frac{t_1}{2\xi_1}\,.\label{xi3}
\end{equation}

In fact $\xi_3$ does not depend on $t_1$. Indeed, if we introduce
\begin{equation}
\eta=1-2a(M-\xi_3)\,,\label{eta}
\end{equation}
equations~\eqref{xi21} and~\eqref{xi22} turn into
\begin{align}
\xi_1 + t_1 a \eta&=0\label{xi31}
\\
\xi_1^2 + t_1^2 a'\big(\eta-\half\big)&=0\label{xi32}\,,
\end{align}
whence
\begin{equation}
a'\big(\eta-\half\big)+a^2\eta^2=0\,.\label{eta1}
\end{equation}

Since $a$ does not depend on $t_1$, the function $\eta$ (as a solution of this
equation) also depends only on $t_2$. Soon we shall show that $x_3$ does not
depend on $t_2$ either, it is a constant function!

But now we remove $\xi_3$ temporary from the play. From~\eqref{eta} we have
\begin{equation}
\xi_3=M-\frac{1-\eta}{2a}\,,\label{xi3a}
\end{equation}
and~\eqref{xi3} together with~\eqref{xi31} yield
\begin{equation}
\xi_3=t_2-\frac1{2a\eta}\,.\label{xi3b}
\end{equation}
The resulting equation is
$$
t_2-\frac1{2a\eta}=M-\frac{1-\eta}{2a}\,,
$$
or
\begin{equation}
t_2-M=\frac{\eta^2-\eta+1}{2a\eta}\,.\label{eta2}
\end{equation}
In result, we have two equations~\eqref{eta1} and~\eqref{eta2} for two unknown
functions $a$ and $\eta$ of the variable $t_2$. Of course, now we could solve
quadratic equation~\eqref{eta1} with respect to $\eta$ and plug the solution
in~\eqref{eta2} trying to solve the resulting differential equation with
respect to $a$. This is possible, but it leads to many complications. Much
simpler is to differentiate~\eqref{eta2}
$$
1=\frac{a\eta'(\eta^2-1)-a'(\eta^3-\eta^2+\eta)}{2a^2\eta^2}
$$
and then to replace the denominator by using~\eqref{eta1}
$$
a'(2\eta-1)=a\eta'(\eta^2-1)-a'(\eta^3-\eta^2+\eta)\,,
$$
or
\begin{equation}
\label{aeta} (\eta^2-1)(\eta-1) a' = a(\eta^2-1)\eta'\,.
\end{equation}

First we consider the possible constant solutions $\eta(t_2)=\pm1$. Solution
$\eta=1$ is not suitable. Indeed, if $\eta=1$, then from~\eqref{eta2} we get
$$
a=\frac1{2(t_2-M)}\,,
$$
whence
\begin{equation}
\label{ADcalc}
A=a^2=\frac1{4(t_2-M)^2}\,,\,\, D=\half a-Ma^2=\frac{t_2-2M}{4(t_2-M)^2}\,,
\end{equation}
\begin{equation}
\label{ADprimecalc}
A'=-\frac1{2(t_2-M)^3}\,,\,\, D'=-\frac{t_2-3M}{4(t_2-M)^3}\,.
\end{equation}
Therefore, in this case, the extremal line $L_{t_1,t_2}$ (cf.~\eqref{linesAD}
has the form
$$
\begin{cases}
\displaystyle x_1 + \frac{t_1x_3}{2(t_2-M)^2} + t_1\frac{t_2-2M}{2(t_2-M)^2}=0
\\
\displaystyle x_2 - \frac{t_1^2x_3}{2(t_2-M)^3} -
t_1^2\frac{t_2-3M}{4(t_2-M)^3}=0\,.
\end{cases}
$$
But the only point of this line with $x_3=M$ belongs to our domain $\Omega$,
because for other points we have
$$
x_2-x_1^2=\frac{t_1^2(2x_3+t_2-3M)}{4(t_2-M)^3}-\frac{t_1^2(x_3+t_2-2M)^2}{4(t_2-M)^4}
=-\frac{t_1^2(x_3-M)^2}{4(t_2-M)^4}<0\,.
$$

Solution $\eta=-1$ is also impossible.
Therefore, we
come to
\begin{equation}
\label{ab1} (\eta-1) a' =  a\eta'\,,
\end{equation}
which implies $a= C\cdot (\eta-1)$. Then~\eqref{xi3a} yields
\begin{equation}
\label{x3} \xi_3(t_2) = M + \frac1{2C} =: c\,.
\end{equation}
Recall that $\xi_3(t_1,t_2)=c$ is the level on which the line $L_{t_1, t_2}$
intersects the boundary $x_2= x_1^2$, $0\le x_3 \le 1$, i.e., the third
coordinate of the point of intersection. So, $ m\le c\le M$. By our assumptions
these lines intersect the boundary $x_3=M$, therefore, to foliate the whole
domain $\Omega$ (at least its interior) we need to have $c=m$. So, from now on
we assume $\xi_3=c=m$. Of course, we could consider any $c>m$ and try to find
an additional solution of the Monge--Amp\'ere equation in the subdomain $m\le
x_3\le c$ but we shall not do this. First of all, we do not need to look for
another solution. The better way of doing is to try now to check that the found
function is just the Bellman function of the problem. Would we have some
obstacles in proving we could continue our search. Another argument explaining
why we do not need ``to glue'' our candidate from two parts is the fact that
there is no concave solution in the subdomain $m\le x_3\le c$ being continuous
extension of the found solution. 

Recall that we have the following expression for $B$ (see~\eqref{shortB})
\begin{equation}
\label{shortB1} B = \half t_1x_1+t_2x_2\,.
\end{equation}
Thus, to find an expression for $B$ we need to find $t_1$ and $t_2$ as
functions of a point $x$ running over our domain $\Omega$.

From~\eqref{xi3b} we have
\begin{equation}
\label{t2} t_2 = m+\frac1{2a[1-2a(M-m)]}\,.
\end{equation}
This gives us the desired $a$ as a function of $t_2$, but it is clear that $a$
is more convenient parameter than $t_2$. We shall express all other functions
of $t_2$ in terms of $a$ and look for $a$ as a function of $x\in\Omega$. To
this aim we return to the equations of the extremal lines~\eqref{linesAD}
rewriting them in terms of $a$. From~\eqref{t2} we have
$$
\frac{dt_2}{da}=-\frac{1-4a(M-m)}{2a^2[1-2a(M-m)]^2}
$$
therefore
\begin{equation}
\label{aprime}
a'=-\frac{2a^2[1-2a(M-m)]^2}{1-4a(M-m)}
\end{equation}
and
\begin{equation}
\label{Aprime}
A=a^2\,,A'=2aa'=-\frac{4a^3[1-2a(M-m)]^2}{1-4a(M-m)}\,,
\end{equation}
\begin{equation}
\label{Dprime}
D =\half a-Ma^2\,, D'=\half
a'(1-4aM)=-\frac{a^2(1-4aM)[1-2a(M-m)]^2}{1-4a(M-m)}\,.
\end{equation}

Now equation~\eqref{linesAD} of the line $L_{t_1, t_2}$ becomes
\begin{equation}
\label{line}
\begin{cases}
\displaystyle x_1 = -t_1a\,[1-2a(M-x_3)]\,,\rule[-15pt]{0pt}{15pt}
\\
\displaystyle x_2 =
t_1^2a^2\frac{[1-2a(M-m)]^2[1-4a(M-x_3)]}{1-4a(M-m)} \,.
\end{cases}
\end{equation}
As we have known, the side surface $\{x_2=x_1^2\}$ is intersected by
$L_{t_1,t_2}$ at the points of the ``bottom lid'' $\{x_3=m\}$
\begin{equation}
\xi=(\xi_1,\xi_1^2,m)\,,\qquad \xi_1=-t_1a\,[1-2a(M-m)]\,.\label{xi}
\end{equation}
If we take $\xi_1$ for parameterizing the extremal lines, then their equations
will be more symmetric. Indeed,~\eqref{xi} implies
\begin{equation}
\label{Bone} t_1=-\frac{\xi_1}{1+2a(M-m)}\,,
\end{equation}
and therefore~\eqref{line} turns into
\begin{equation}
\label{line1}
\begin{cases}
\displaystyle x_1 = \frac{1-2a(M-x_3)}{1-2a(M-m)}\,\xi_1\,,
\rule[-20pt]{0pt}{20pt}
\\
\displaystyle x_2 = \frac{1-4a(M-x_3)}{1-4a(M-m)}\,\xi_1^2\,.
\end{cases}
\end{equation}

These equations immediately supply us with an expression for $a$, namely,
\hbox{$a=a(x)$} is a root of the following cubic equation
\begin{equation}
\label{cubicc} s:= \frac{x_1^2}{x_2} = \left[  \frac{1-2a(M-x_3)}{1-2a(M-m)}
\right]^2 \frac{1-4a(M-m)}{1-4a(M-x_3)}\,.
\end{equation}

To determine, which of three possible roots of the equation~\eqref{cubicc}
gives us the desired value of $a$, we investigate the above function $s=s(a)$
defined by~\eqref{cubicc} as a function of the parameter $a$, all other
parameters assuming to be fixed.

First of all we note that the extremal line~\eqref{line1} intersects the plane
$\{x_3=M\}$ at the point
\begin{equation}
\zeta=(\zeta_1,\zeta_2,M)\,,\qquad \zeta_1= \frac{\xi_1}{1-2a(M-m)}\,, \qquad
\zeta_2 = \frac{\xi_1^2}{1-4a(M-m)}\,.
\label{zeta}
\end{equation}
Since $\zeta_2\ge0$, we have the first restriction for $a$:
\begin{equation}
\label{doma} a<\frac1{4(M-m)}\,.
\end{equation}
The behavior of $s(a)$ on the semi-axis~\eqref{doma} is the same for all values
of other parameters: on the negative half-line it monotonously increases from
$\frac{M-x_3}{M-m}$ till $1$, and on the interval $[0,\frac1{4(M-m)}]$ it
monotonously decreases from $1$ to $0$. Therefore, for a given $s$, $0\le
s\le1$, we have one or two solutions of~\eqref{cubicc}
satisfying~\eqref{doma}: for all $s$ there exists a positive solution $a=a^+$
and for $s>\frac{M-x_3}{M-m}$ there exists a negative solution $a=a^-$ as well.
We shall see that our solution is $a=a^+$, but the solution $a=a^-$ is not
meaningless: unexpectedly we get a solution of another extremal problem,
namely~\eqref{cetbelmin}.

To complete step II of constructing a Bellman function candidate we need to
write down formula~\eqref{eqcet1} or~\eqref{eqcet1min}. Everything is ready to
do this: we have expression~\eqref{shortB1} for $B$, where we need to
substitute~\eqref{t2} for $t_2$ and to take $t_1$ from the first line
of~\eqref{line}. So, we get
\begin{align}
B(x;m,M)&=-\frac{x_1^2}{2a\,[1-2a(M-x_3)]}+\left(m+\frac1{2a[1-2a(M-m)]}\right)x_2
\label{firstB}
\\
&=\frac{(x_3-m)x_2}{\bigl[1-2a(M-m)\bigr]^2\bigl[1-4a(M-x_3)\bigr]}+mx_2\,,
\label{secondB}
\end{align}
where $a=a^+$ is the positive root ($4(M-m)a^+\le 1$) of the cubic
equation~\eqref{cubicc} for Theorem~\ref{bc}, and $a=a^-$ is the negative one
for Theorem~\ref{bcmin}.

\end{proof}

\section{Foliation according to $B$.}
\label{foli}

We already saw that lines $L_{t_1, t_2}$ are given by equations~\eqref{line1}
and in the parametrization of the family of the extremal lines it is more
convenient to use parameter $a$ instead of $t_2$ and $\xi_1$ instead of $t_1$:
$$
\ell_{a,\xi_1} := L_{t_1, t_2}\,.
$$
For a given $\xi_1$ we have a ``fan'' of extremal lines starting at an
arbitrary point $\xi=(\xi_1,\xi_1^2,m)$ on the ``edge'' $\{x\colon
x_2=x_1^2,\,x_3=m\}$:
\begin{align*}
F_{\xi_1}^+&=\big\{\ell_{a,\xi_1}\colon 0\le a<\frac1{4(M-m)}\big\}\,,
\\
F_{\xi_1}^-&=\big\{\ell_{a,\xi_1}\colon -\infty<a\le0\big\}\,,
\end{align*}
where $F_{\xi_1}^+$ is the fan of extremal lines of $B_{\max}$ and
$F_{\xi_1}^-$ is the fan of extremal lines of $B_{\min}$.

Let us find the trace of our fan on the boundary  $x_3=M$. We remember that the
point of intersection $\zeta$ is given by~\eqref{zeta}:
$$
\begin{cases}
\zeta_1=&\displaystyle\!\!\frac{\xi_1}{1-2a(M-m)}\,,
\\
\zeta_2=&\displaystyle\!\!\frac{\xi_1^2}{1-4a(M-m)}\,.\rule{0pt}{25pt}
\end{cases}
$$
This is a parametric equations of the hyperbola $C_{\xi_1}$
\begin{equation}
\label{curve} \zeta_2=\frac{\zeta_1\xi_1^2}{2\xi_1 -\zeta_1}\,,
\end{equation}
tangent to the parabola $\zeta_2=\zeta_1^2$ at the point $(\xi_1, \xi_1^2)$.

The intersection of the fan $F_{\xi_1}^+$ with the plane $x_3=M$ is the piece
$C_{\xi_1}^+$ of this hyperbola between the points $\xi_1$ and $2\xi_1$, for
the fan $F_{\xi_1}^-$ it is the piece $C_{\xi_1}^-$ of this hyperbola between
$\xi_1$ and zero. For $\xi_1=0$ it is the axis $C_0^\pm=\{(0,\zeta_2)\}$.

Everything is on the ``upper lid" $\{x_3 =M, x_1^2 \le x_2\}$ of $\Omega$, and
the curves $C_{x_1}^+$, as well as the curves $C_{x_1}^-$, foliate this ``upper
lid". Lines $\ell_{a, \xi_1}\in F_{\xi_1}^\pm$ connect the points of
$C_{\xi_1}^\pm$ to the point $(\xi_1, \xi_1^2, m)$ on the boundary of $\Omega$
and each of two sets of lines $ \{F_{\xi_1}^\pm\colon-\infty<\xi_1<+\infty\}$
foliates some subdomain of $\Omega$. The lines of the fan  $F_{\xi_1}^+$
foliate the whole $\Omega$, whereas  $F_{\xi_1}^-$ foliate the subdomain
\begin{equation}
\big\{x=\{x_1,x_2,x_3\}\colon\frac{M-x_3}{M-m}<\frac{x_1^2}{x_2}\le1,\ x_3\le M
\big\}\,, \label{omega-min}
\end{equation}
where the function $B_{\min}$ is defined. Indeed, we have a solution $a=a^+$
of~\eqref{cubicc} for arbitrary point of $\Omega$, while a solution $a=a^-$
exists if and only if the point $x$ is from~\eqref{omega-min}.

\section{Lower boundary of $\Omega$.}

Notice that we found our function $B$ only in $\{x_1^2 \le x_2,\ m<x_3\le M\}$
for $B_{\max}$ and in $\{x_1^2 \le x_2,\ M-(M-m)\frac{x_1^2}{x_2}<x_3\le M\}$
for $B_{\min}$. However both these functions are continuous in the closed
domains. For $B_{\min}$ clearly the limit is $mx_2$ as $x_3\to
M-(M-m)\frac{x_1^2}{x_2}$, what corresponds to $a\to-\infty$. For $B_{\max}$
formula~\eqref{secondB} has indeterminancy $\frac{0}{0}$ when $x_3\to m$, what
corresponds $a\to\frac1{4(M-m)}$. But we can easily pass to the limit
in~\eqref{firstB} and obtain by continuity the values of the function
$B_{\max}(x_1, x_2, m; m, M)$ on the ``lower lid". Namely,
\begin{equation}
\label{ll} B_{\max}(x_1, x_2, m; m, M) = 4(M-m)(x_2 - x_1^2) + m x_2\,.
\end{equation}

\medskip

\noindent{\bf The norm of embedding becomes apparent.}
In fact, the best constant in the inequality
$$
B_{max} \le C\, x_2
$$ can be read from \eqref{ll}, and it is indeed $4(M-m)$. This can be easily proved, but notice that it is not even explicit in the formula for $B_{max}$!

\section{Reducing parameters $M$ and $m$.}

Before we start to prove Theorems~\ref{bc} and~\ref{bcmin}, we would like to
show that it is sufficient to prove them for $m=0$ and $M=1$. On the one hand,
we have
\begin{equation}
\Bel(x_1,x_2,x_3;m,M)=\Bel(x_1,x_2,\frac{x_3-m}{M-m};0,1)+mx_2.
\label{reduce-mM}
\end{equation}
This is direct consequence of definition. Indeed, while the homogeneity
condition~\eqref{hmB} follows from the definition if to compare a set of test
functions $\{\phi\colon \av{\phi}I=x_1,\,\av{\phi^2}I=x_2\}$ and the set
$\{\tilde\phi=t\phi\}$ with the same measure $\mu$ and the same set of point
masses $\alpha_J$, relation~\eqref{reduce-mM} follows from considering the same
set of test functions $\phi$ but comparing with a ``renormalized'' measure
$\tilde\mu(J)=\frac{\mu(J)-m|J|}{M-m}$ and point masses
$\tilde\alpha_J=\frac{\alpha_J}{M-m}$, then we shall have $\tilde m=0$, $\tilde
M=1$, and $\tilde x_3=\frac{x_3-m}{M-m}$.

\medskip

Notice that function $\Bel(x)$ from~\eqref{Bel} is defined differently than
$\Bel_{CET}(x;0,1)$. Its definition does not allow for measure $\mu$ on the
boundary. However, these functions turned out to be equal, extremal $\mu$ must
be zero. It is not clear how to see this immediately from the
definitions~\eqref{Bel} and~\eqref{cetbel} of $\Bel(x)$ and $\Bel_{CET}(x;0,1)$
correspondingly. However, since the supremum in the definition of $\Bel(x)$ is
taken over the smaller set of test measures ($\mu=0$), we have the inequality
$\Bel(x)\le\Bel_{CET}(x;0,1)$. So, the proof of Theorems~\ref{bc01}
and~\ref{bc} will consist in proving two inequalities: $\Bel_{CET}(x;0,1)\le
B(x;0,1)$ and $\Bel(x)\ge B(x;0,1)$.

\medskip

Let us introduce the notations:
\begin{equation}
\label{bece}
B_c(x):= B(x;1-c,1)\,,\Bel_c(x):=\Bel_{CET}(x;1-c,1)\,\, c\in (0,1]\,.
\end{equation}

Obviously we can rescale everything and consider only $B_c$ and $\Bel_c$.
Moreover, it is enough to consider only $c=1$.

In fact, our Bellman function candidate $B(x;m,M)$ clearly satisfies
relation~\eqref{reduce-mM}, therefore, to prove $\Bel_{CET}(x;m,M)=B(x;m,M)$ it is sufficient to check that
$$
\Bel_{CET}(x;0,1)=B(x;0,1)
$$
for all suitable arguments $x$. We will first prove that
$$
B(x)\ge \Bel(x;0,1)\,.
$$
 Then we prove

 $$
 B\le \Bel\,.
 $$
 Then the obvious inequality
 $$
 \Bel\le\Bel(x;0,1)
 $$
 finishes the proof.

\vskip1cm

\section{$B_c(x)\ge \Bel_c(x)$. Concavity.}
\label{Bellsm}

In what follows symbol $B(x)$ stands always for $B(x;0,1)$
from~\eqref{eqcet101} (which is the same as $B_c(x)$ with $c=1$).

\begin{theorem}
\label{conc}
Function $B_c$ is concave in the domain $\Omega_c$.
\end{theorem}

\begin{proof}

It is enough to prove this for $c=1$. Then rescaling proves the rest. In what
follows  $a$ is always the unique root of \eqref{cubicc} (with $M=1, m=0$)
lying in $[0,\frac1{4})$. In previous sections we calculated ($c=1$):
\begin{eqnarray}
\label{t1t2t3}
\Bone = t_1 = -\frac{x_1}{a[1-2a(1-x_3)]}\,\\
\Btwo = t_2 =  \frac{1}{2a(1-2ca)}\,\\
\Bthree = t_3 = \frac{x_1^2}{[1-2a(1-x_3)]}\,.
\end{eqnarray}

We can now compute Hessian $d^2\,B$, where (we use the notation $s=x_1^2/x_2$)
$$
d^2\,B(x): = M(x)=
$$
$$
 \begin{bmatrix}
 -\frac{1}{a[1-2a(1-x_3)]} +\frac{1-4a(1-x_3)}{a^2[1-2a(1-x_3)]^2}\cdot \frac{2x_1^2}{x_2}\cdot\frac{\partial a}{\partial s}\,,  & -\frac{1-4a}{a^2(1-2a)^2}\cdot \frac{x_1}{x_2}\cdot\frac{\partial a}{\partial s}\,,  &  -\frac{2x_1}{[1-2a(1-x_3)]^2}-\frac{x_1[1-4a(1-x_3)]}{a^2[1-2a(1-x_3)]^2} \cdot \frac{\partial a}{\partial x_3}\\
 -\frac{1-4a(1-x_3)}{a^2[1-2a(1-x_3)]^2} \cdot \frac{x_1^3}{x_2^2}\cdot\frac{\partial a}{\partial s}\,,  &  \frac{1-4a}{2a^2(1-2a)^2}\cdot \frac{x_1^2}{x_2^2}\cdot\frac{\partial a}{\partial s}\,,  &  -\frac{1-4a}{2a^2(1-2a)^2}\cdot \frac{\partial a}{\partial x_3}\\
  \frac{2x_1}{[1-2a(1-x_3)]^2} + \frac{1-x_3}{[1-2a(1-x_3)]^3}\cdot \frac{8x_1^3}{x_2}\frac{\partial a}{\partial s} &  -\frac{1-x_3}{[1-2a(1-x_3)]^3}\cdot\frac{4x_1^4}{x_2^2}\cdot\frac{\partial a}{\partial x_3} & -\frac{4x_1^2a}{[1-2a(1-x_3)]^3}+ \frac{4x_1^2(1-x_3)}{[1-2a(1-x_3)]^3}\cdot \frac{\partial a}{\partial x_3}
\end{bmatrix}
$$

\medskip

The element $M_{12}$ actually is equal to $M_{21}$, $M_{13}$ is equal to
$M_{31}$  by \eqref{cubicc}.
The reader may try to prove that this matrix is non-positive for every
$x\in\Omega$ by direct calculation. We prefer an oblique way of doing that. Put

$$
N(x) := \begin{bmatrix}
\displaystyle -\frac{1}{a[1-2a(1-x_3)]}+\frac{1-4a(1-x_3)}{a^2[1-2a(1-x_3)]^2}\cdot \frac{2x_1^2}{x_2}\cdot\frac{\partial a}{\partial s}\,,   \,\,  &    \displaystyle -\frac{1-4a}{a^2(1-2a)^2}\cdot \frac{x_1}{x_2}\cdot\frac{\partial a}{\partial s}\\
\displaystyle -\frac{1-4a(1-x_3)}{a^2[1-2a(1-x_3)]^2} \cdot
\frac{x_1^3}{x_2^2}\cdot\frac{\partial a}{\partial s}\,,  &
\displaystyle \frac{1-4a}{2a^2(1-2a)^2}\cdot \frac{x_1^2}{x_2^2}\cdot\frac{\partial
a}{\partial s}
\end{bmatrix}
$$

$$
L(x) := \begin{bmatrix}
\displaystyle \frac{1-4a(1-x_3)}{a^2[1-2a(1-x_3)]^2}\cdot \frac{2x_1^2}{x_2}\cdot\frac{\partial a}{\partial s}\,,   \,\,  &    \displaystyle -\frac{1-4a}{a^2(1-2a)^2}\cdot \frac{x_1}{x_2}\cdot\frac{\partial a}{\partial s}\\
\displaystyle -\frac{1-4a(1-x_3)}{a^2[1-2a(1-x_3)]^2} \cdot
\frac{x_1^3}{x_2^2}\cdot\frac{\partial a}{\partial s}\,,  &
\displaystyle \frac{1-4a}{2a^2(1-2a)^2}\cdot \frac{x_1^2}{x_2^2}\cdot\frac{\partial
a}{\partial s}
\end{bmatrix}
$$

Let us prove that
$$
M_{11}<0\,,\,\,L\le 0\,,\,\,N \le 0\,.
$$

We know that  function $s\rightarrow a(s)$ decreases from $\frac1{4}$ to $0$
when $s$ goes from $0$ to $1$. Therefore, $a'\ge 0$. Also $1-4a(1-x_3)= (1-4a)
+ 4a x_3 \ge 0$ because $a\le\frac1{4}$, $x_3 \ge 0$, and $a>0$.
Therefore, $L_{11}\le0, N_{22}= L_{22}\le 0, N_{11}<0$ for every $x\in
\Omega\cap \{0<x_3\}$. In particular, $M_{11} <0$.

On the other hand, it  is immediate to see that $L_{11} L_{22}-L_{12}L_{21}$ vanishes identically.
Then we can see right away that $N\le 0$ as well, but more than that
\begin{equation}
\label{detN} N< 0\,,\,\, \det N (x) \neq 0\,\,\,\,\forall x\in
\Omega\cap\{0<x_3\}\cap\{x_1\neq 0\}\,.
\end{equation}
In fact, it follows because $N= L + \text{diag}\{\displaystyle -\frac1{a[1-2a(1-x_3)]}\,, 0\}$
if we notice that
$$
-\frac1{a[1-2a(1-x_3)]}<0
$$
and $L_{22} \le 0$.

We would like to recall the reader that we obtained $B$ of
\eqref{firstB}--\eqref{secondB} just by solving the Monge--Amp\`ere  equation
$\det d^2B =0$ in $\Omega$. So, of course, $\det M(x)=0$. However, we
propose to the reader to check this as follows.

Fix a point $x\in\Omega\,, 0< x_3$, and consider the line $L_{t_1,t_2}$
passing through $x$. Its directional vector was computed, it is (see
\eqref{direction}, \eqref{AD}, \eqref{Aprime}, and \eqref{line} combined):
\begin{equation}
\label{d} d(x)= (\frac{2ax_1}{1-2a(1-x_3)}\,; \frac{4a(1-2a)^2
x_1^2}{(1-4a)[1-2a(1-x_3)]^2}\,; 1)^{T}\,.
\end{equation}

Actually we {\em built} $B$ from the condition that $d^2B$ annihilate the
vector field $d(x), x\in \Omega$. Also it is  easy to see that $M(x) d(x)  = 0$ from
the direct calculation.

We saw that for every $x\in \Omega\cap \{0< x_3\}$  we have $M_{11} <0,
M_{11}M_{22} - M_{12}^2 \le 0, \det M(x) =0$. By a well-known fact from linear
algebra we conclude that $d^2\,B=M(x)$ is negatively defined for all such $x$.
Concavity of $B$ is fully proved.

\end{proof}

\begin{theorem}
\label{Bx3}  In $\Omega$ we have $\Bthree(x) \ge x_1^2$. The same is true for $B_c$ in $\Omega_c$ for $c\in(0,1]$.
\end{theorem}

\begin{proof}
The second claim follows from the first by rescaling. The first claim is
obvious from \eqref{t1t2t3}:  $\displaystyle\Bthree
=\frac{x_1^2}{[1-2a(1-x_3)]^2}\ge x_1^2$ because $0<1-2a(1-x_3)\le 1$. We
already noted the left inequality, it follows from $a>0$ and $a\le
\frac{1}{4}\le \frac{1}{2}$. The right inequality is just $a>0$.

\end{proof}

\begin{theorem}
\label{smaller} Let $c\in (0,1]$. In $\Omega_c$ we have $\Bel_c\le B_c$.
\end{theorem}

\begin{proof}
Again it is enough to consider only the case $c=1$. Let us combine
Theorem~\ref{conc} and Theorem~\ref{Bx3} to obtain the following {\bf main
inequality}:

\begin{eqnarray}
\label{MIBC} B(x_1,x_2,x_3) -\half [B(x_1^+,x_2^+,x_3^+) +
B(x_1^-,x_2^-,x_3^-)] \ge x_1^2 [x_3 - \half (x_3^+
+ x_3^-)]\,,\,\,\,\, \\
\forall  x, x^+, x^- \in \Omega\,\,\text{such that}\,\, x_1=\half (x_1^ +
+x_1^-)\,,\,   x_2=\half (x_2^ +  +x_2^-)\,,\, \\ x_3=\half (x_3^ +  +x_3^-)\,.
\end{eqnarray}

In fact, this is easy. Put $x_i(t) = \half [x_i^+ (1+t) + x_i^-(1-t)]\,,\,\,
i=1,2$, and
$$
x_3(t)=\begin{cases} -t\, x_3^-  + x_3 (1+t)\,,\,\,\,\text{if}\,\,\,t\in [-1,0]\,,\\
x_3 (1-t) +t \,x_3^+\,,\,\,\,\text{if}\,\,\,t\in [0,1]\,.
\end{cases}
$$

Set

$$
b(t):= B(x_1(t), x_2(t), x_3(t))\,,\,\, t\in [-1,1]\,.
$$
  Then the {\bf main inequality} above transforms into

\begin{equation}
\label{bt}
b(0) -\half (b(1) + b(-1) \ge x_1(0)^2 [x_3(0)- \half (x_3(1) + x_3(-1)]\,.
\end{equation}
 To prove \eqref{bt} let us notice that  $x_3(t)$ is concave, and negative measure $x''_3$ is the following  $x''_3(t) = -2[x_3(0)- \half (x_3(1) + x_3(-1)]\cdot \delta_{0}$. Then, of course,  $b$ is concave  (see Theorems \ref{conc} and \ref{Bx3}) and measure
 $$
 b''(t) = (d^2B x'(t), x'(t))dt  + \Bthree(x(t)) x''_3(t) = (d^2B(x(t)) \,x'(t), x'(t))dt
 $$
 $$
  - 2 \Bthree(x(0))[x_3(0)- \half (x_3(1) + x_3(-1)]\delta_{0}\,.
 $$

 The following formula finishes the proof:
 $$
 b(0) -\half (b(1) + b(-1) = -\half\int_{-1}^1 (1-|t|) b''(t)\,.
 $$
 In fact, combining the last two formulae we obtain (we use Theorems \ref{conc} and  \ref{Bx3} again)

 \begin{eqnarray*}
 b(0) -\half (b(1) + b(-1) = \half \int_{-1}^1 (1-|t|) (-d^2B(x(t))\, x'(t), x'(t))dt\\  + \Bthree(x(0))[x_3(0)- \half (x_3(1) + x_3(-1)]
 \\\ge x_1(0)^2 [x_3(0)- \half (x_3(1) + x_3(-1)]\,,
 \end{eqnarray*}
 which is desired \eqref{bt}.

 \vspace{.1in}

 Now {\bf the main inequality} \eqref{MIBC} and {\bf the  convexity of the domain} $\Omega$ will allow us to finish the proof of the theorem.
 We need a simple  lemma about subharmonic functions on graphs.

 We consider the dyadic tree $\T$, whose vertices are denoted by $p_{\sigma}$, where $\sigma$ is the word formed by $\pm$, the empty word $\sigma_0$ is the root of the tree, and $|\sigma|$ is the length of the word.
 Function on $\T$ is called superharmonic if  its discrete Laplacian is non-negative
 $$
 \forall \sigma \in \T \,\,\, \Delta f (\sigma) = f(\sigma) -\half (f(\sigma +)+ f(\sigma -)) \ge 0\,.
 $$
 We can associate the boundary of the tree with segment $[0,1]$ in the following sense. For Lebesgue almost every point $x\in [0,1]$ we have a unique
 branch $b(x)=(\sigma_0,\sigma_1(x),...,\sigma_n(x),...)$ of $\T$ associated with it: just consider the dyadic form of $x$ and its $n$-th digit encodes whether to branch to $+$ or $-$ side on the $n$-th stage. Given $f$ on the tree we put
 $$
 F(x) = \liminf_{n\rightarrow \infty} f(\sigma_n(x))\,.
 $$

 Here is a lemma  which deserves to be called { \bf Green's formula for the dyadic tree}

 \begin{lemma}
 \label{tree}
 Given a positive  finite superharmonic  function $f$ on $\T$ we get
 $$
 \int_0^1 F(x)\,dx + \sum_{\sigma \in \T} 2^{-|\sigma|}\Delta f (\sigma) \le f(\sigma_0)\,.
 $$
 \end{lemma}
 \begin{proof}
 Obvious.
 \end{proof}

 Now fix $x= (x_1, x_2, x_3) \in \Omega$ and fix any function $\phi$ on the interval $I=[0,1]$ (we recall that nothing depends on the original interval) with
 \begin{equation}
 \label{initx}
 \langle \phi \rangle_I = x_1\,,\,\, \langle \phi^2 \rangle_I = x_2\,,
 \end{equation}
a non-negative measure $\mu$ on $I$, such that $0\le \mu(J)\le |J|$ for all dyadic subintervals $J$ of $I$, and any collection of non-negative numbers $\{\alpha_J\}_{J\in D(I)}$ such that
 \begin{equation}
 \label{initM}
 \frac{\mu(I)}{|I|}+\frac1{|I|} \sum_{J\in D(I) } \alpha_J = x_3 \in (0, 1]\,,\,\,
 \frac{\mu(\ell)}{|\ell|}+ \frac1{|\ell|} \sum_{J\in D(\ell) } \alpha_J \in [0, 1]\,\,\, \forall \ell\in D(I)\,.
 \end{equation}
 We immediately see that $\mu= w\,dx$, where $0\le w\le 1$.
 
 Intervals of $D(I)$ and dyadic tree $\T$ are in natural one to one correspondence. We call $I_{\sigma}$ the interval corresponding to vertex $\sigma$, $I$ corresponds to $\sigma_0$. Consider the following function. Take $B$ from \eqref{firstB} (with $M=1, m=0$) and put
 $$
 M(\sigma) := \frac1{|I_{\sigma}|}(\mu(I_{\sigma}) + \sum_{J\in D(I_{\sigma}) } \alpha_J)
 $$
  and
 $$
 f(\sigma) := B ( \langle \phi \rangle_{I_{\sigma}}, \langle \phi^2 \rangle_{I_{\sigma}}, M(\sigma))\,.
 $$
 Then of course
 \begin{equation}
 \label{init}
 f(\sigma_0) = B ( x_1, x_2, x_3)\,.
 \end{equation}
 By the main inequality \eqref{MIBC} it is a superharmonic function on the tree $\T$: more than that, we can estimate its discrete Laplacian from below. By this same \eqref{MIBC}
 $$
 \Delta f (\sigma) \ge  \langle \phi \rangle_{I_{\sigma}}^2 \Delta M(\sigma)\,.
 $$

 It is immediately seen that

 $$
 \Delta M(\sigma) =  \alpha_{I_{\sigma}}/|I_{\sigma}|\ge 0\,.
 $$
 Therefore, combining the last two inequalities we get
 $$
 \frac{2^{|\sigma|}\langle \phi \rangle_{I_{\sigma}}^2 \alpha_{I_{\sigma}}}{|I_{\sigma_0}|} \le \Delta f(\sigma)\,.
 $$

One can observe that $B(x_1,x_2,x_3) \ge x_2 x_3$. Then
$$
f(\sigma) \ge \langle\phi^2\rangle_{I_{\sigma}} M(\sigma) \ge \langle\phi^2\rangle_{I_{\sigma}} \cdot\frac{\mu(I_{\sigma})}{|I_{\sigma}|}\,.
$$

Then for almost every $x\in [0,1]$ we denote by $\sigma(x)$ the branch landing at it and by $\sigma_n(x)$ the $n$-th vertex on this branch. Then for a. e. $x$
$$
\liminf_{n\rightarrow \infty} f(\sigma_n(x)) \ge \lim_{n\rightarrow\infty}\langle\phi^2\rangle_{I_{\sigma}} \cdot\frac{\mu(I_{\sigma})}{|I_{\sigma}|}= \phi^2(x) w(x)\,.
$$

 Recall that $I=I_{\sigma_0}=[0,1]$. It is time to use Lemma \ref{tree}, the last inequalities, and Fatou's lemma (which is $\int\liminf_n f_n\le \liminf_n\int f_n$ for a sequence of non-negative functions $f_n$):

 \begin{equation}
 \label{sum} \frac{1}{|I_{\sigma_0}|}\int_{I_{\sigma_0}}\phi^2(x)\,w(x)dx+\frac{1}{|I_{\sigma_0}|}\sum_{\sigma \in \T} \langle \phi \rangle_{I_{\sigma}}^2 \alpha_{I_{\sigma}} \le B ( x_1, x_2, x_3)\,.
 \end{equation}
 We used here our lemma and \eqref{init}.

 But inequality \eqref{sum} proves that $\Bel(x; 0,1)\le B(x;0,1)$.  In fact, in the left hand side of \eqref{sum} we have arbitrary $\phi$ satisfying \eqref{initx} and arbitrary numbers $\alpha_{\cdot}$ satisfying \eqref{initM}. Function $\Bel(x_1, x_2, x_3)$ by definition is the supremum of such sums over all such functions and collections of numbers. Therefore, inequality $B_c(x) \ge \Bel_c(x)$ is completely proved and so is the theorem.

\end{proof}

\section{$B_c(x) \le \Bel_c(x)$. Extremal sequences.}
\label{Bellbig}

\begin{theorem}
\label{smaller} Let $c\in (0,1]$. In $\{x_1^2 \le x_2\,,\,\, 1-c <x_3 \le 1$ we
have $\Bel_c\ge B_c$.
\end{theorem}

\begin{proof}
Below $I_0=[0,1]$. We consider first only $c=1$ and  $x= (x_1, x_2, 1)$ on the
``upper lid" of $\Omega$. We want to fix a large integer $n$ and to  construct
function $\phi_n$ and sequence $\{\alpha^n_{\ell}\}_{\ell\in D(I_0)}$ in such a
way that the sum
$$
\sum _{\ell\in D(I_0)} \langle \phi_n\rangle_{\ell}^2 \alpha^n_{\ell} > B(x_1,
x_2, 1) - \varepsilon_n\,,\,\,\,\,\,\varepsilon_n\rightarrow 0\,.
$$
where
\begin{equation}
\label{x1} \langle \phi_n\rangle_{I_0} = x_1
\end{equation}
and
\begin{equation}
\label{x2} \langle \phi_n^2\rangle_{I_0} = x_2\,.
\end{equation}
 Fix a large integer $n$ and split $I_0= [0,1]$ into
the union of $n+1$ intervals $I_{k}=[2^{-k}, 2^{-k+1}]\,,\,\, 1\leq k \le n$,
and $J =[0, 2^{-n}]$. Maps $m_k(x) = 2^{k}x -1$ map $I_{k}\,,\,\,1\leq k \le n$
onto $I_0$. Consider preimages of $I_s$, $s=1,...,n$, under $m_k, k=1,...,n$.
We obtain $I_{ks}$.  Then $m_s \circ m_k$ maps $I_{ks}$ onto $I_0$. We also
consider $J_k$, $k=1,...,n$, preimages of $J$ under $m_k$. We iterate this
procedure obtaining $I_{k_1k_2...k_m}$, $J_{k_1k_2...k_{m-1}}$. Put
$\alpha^n_{I_k}:= 2^{-k}\cdot 2^{-n}= |J_k|\,,\,\, 0\le k\le n$. We put
$\alpha^n_{k_1k_2...k_m}=|J_{k_1k_2...k_m}|$.  For all other dyadic intervals we
put $\alpha_I=0$. We call $J$ the $0$ generation, and all $J_{k_1k_2...k_m}$
the $m$-th generation. Obviously sum of lengths of all generations of $J$'s is
equal to $1$. There for we have proved
$$
\sum_{\ell \in D(I_0)} \alpha_{\ell} =1\,.
$$
Moreover, it is easy to see that all $J$'s are disjoint (except may be the end-points) so
$$
\sum_{\ell \in D(I_0)\,,\,\ell\subset I} \alpha_{\ell} \le |I|\,\,\forall I\in D(I_0)\,.
$$
Now we are going to define $\phi_n$.
\begin{equation}
\label{phin}
\phi_n(x) =\begin{cases}
c_n x_1\,\,\,\,\,x\in J\\
\phi_n(m_k(x))\,\,\,\, x\in I_k\,\,\, k=2,...,n\\
d_n \phi_n(m_1(x))\,\,\,x\in I_1\,.
\end{cases}
\end{equation}
Here $c_n$ and $d_n$ are constants which we will define now. First of all
notice that this recursive definition of $\phi_n$ really defines it a.e. (as
long as $c_n, d_n$ are prescribed) just because we have it already defined on
the $0$ generation of $J$'s (that is on $J$ itself). But then it is defined on
$m_k$-pre-images of $J$ (so on the first generation of $J$'s), but then it is
defined on pre-images of pre-images (second generation of $J$'s), et cetera...
But The union of all generations of $J$'s gives us $I_0$ up to a set of zero
Lebesgue measure.

We define $c_n, d_n$ from requirements \eqref{x1}, \eqref{x2}:
\begin{eqnarray*}
\label{cndn}
\frac1{2^n}c_n x_1 + (1-\frac1{2^n} -\frac12) x_1 +\frac12 d_n x_1 = x_1\\
\frac1{2^n}c_n^2 x_2 + (1-\frac1{2^n} -\frac12) x_2 +\frac12 d_n^2 x_2 = x_2\,.
\end{eqnarray*}

Then we get (using our notations $s=x_1^2/x_2$ and choosing proper root of
quadratic equation)
\begin{equation}
\label{cn} c_n =\frac{(2^{n-1} +1) 1/s - \sqrt{2^{n-1}(2^{n-1} +1) (x_1^2-
1/s)}}{2^{n-1} + 1/s} \rightarrow 1/s - \sqrt{ 1/s^2 - 1/s}\,.
\end{equation}

\begin{lemma}
\label{geql}
We have
\begin{equation}
\label{geq}
\phi(x)c_n\chi_{J_{k_1...k_m}} \leq c_n\langle \phi_n\rangle_{I_{k_1...k_m}}\,,\,\,\,\forall x\in J_{k_1...k_m}\,,
\end{equation}
where empty sequence $k_1...k_m$ corresponds to $I_0$ and $J$ correspondingly.
\end{lemma}

\begin{proof}
Let us start with the empty sequence. Then we should check that
$$
\phi(x) \leq c_n\langle \phi_n\rangle_{I_{0}}  \,,\,\,\,\forall x\in J\,.
$$

But the average over $I_0$ is $x_1$, and $\phi(x) = c_n x_1$ on $J$ by
definition. The rest of the lemma follows from the self-similarity of $\phi_n$,
in fact, it is either the same function $\phi_n$ ``pre-shrunk'' to a smaller
interval, or it is a  fixed multiple of this. So we are done.
\end{proof}

Let us now square \eqref{geq} and integrate it. Then sum over all $J$'s of all generations.
We will get from \eqref{cn}
$$
\sum_{\ell\in D(I_0)}\langle \phi_n\rangle^2 \alpha^n_{\ell} \ge
\frac{x_2}{c_n^2} = x_1^2\frac{1/s}{(1/s - \sqrt{ 1/s^2 - 1/s})^2}
-\varepsilon_n\,.
$$
But the first term is
\begin{eqnarray*}
\frac{ x_1^2}{(\sqrt{1/s} - \sqrt{ 1/s - 1})^2}  = x_1^2 (\sqrt{1/s} +\sqrt{
1/s - 1})^2\\=
x_1^2\Big(\sqrt{\frac{x_2}{x_1^2}}-\sqrt{\frac{x_2}{x_1^2}-1}\Big)^2=
(\sqrt{x_2} + \sqrt{x_2-x_1^2})^2\,.
 \end{eqnarray*}
We finally get
\begin{equation}
\label{geq1} \sum_{\ell\in D(I_0)}\langle \phi_n\rangle^2 \alpha^n_{\ell} \ge
(\sqrt{x_2} + \sqrt{x_2-x_1^2})^2 -\varepsilon_n\,.
\end{equation}

Let us now look at formula \eqref{secondB} with $x_3=1$ and plug $a$ from
\eqref{cubicc} with $x_3=1$. Readily, \eqref{cubicc} with $x_3=1$ becomes
$\frac{1+4a}{(1+2a)^2} =\frac{x_1^2}{x_2}$. From here and from negativity of
$a$ we obtain
$$
\frac1{1+2a} = 1 + \sqrt{\frac{x_2-x_1^2}{x_2}}\,.
$$
But formula \eqref{secondB} with $x_3=1$  (and $c=1$) gives
$$
B(x_1, x_2, 1) =\Big( \frac1{1+2a}\big)^2 x_2 =\Big( 1 +
\sqrt{\frac{x_2-x_1^2}{x_2}}\Big)^2 x_2 = (\sqrt{x_2} + \sqrt{x_2-x_1^2})^2 \,.
$$
Using \eqref{geq1} we finally we get
\begin{equation}
\label{geq2} \sum_{\ell\in D(I_0)}\langle \phi_n\rangle^2 \alpha^n_{\ell} \ge
B(x_1, x_2, 1)  -\varepsilon_n\,,
\end{equation}
and our theorem is proved for $x_3=1, c=1$.

If $c=1$ but $0< x_3 <1$, then point $x=(x_1, x_2, x_3)\in \Omega$ lies on a
certain line $L_{t_1, t_2}$. Function $B$ is linear on this line. So combining
constant function and function $\phi_n$ we will obtain the general inequality.

For $c\in (0,1)$ the proof is actually exactly the same.
\end{proof}

\subsection{Geometry of foliation: explanation of the choice of $\phi_n$ in \eqref{phin}.}
\label{Phin}

\bigskip

This section contains a heuristic explanation why function $\phi_n$ was chosen
in this form in~\eqref{phin}.

The main point here is a certain geometric observation concerning the foliation
of the domain $\Omega$ by lines $L_{t_1t_2}$. Consider the foliation of the
upper lid $\{X_3=1\}\cap\Omega$ by parabolas $P_A:= \{ X_2 = A\,X_1^2, \,
X_3=1\}\,\,, A\ge 1$. Then we have the following easy geometric observation.

\begin{theorem}
\label{foliupper} Let $L_{t_1,t_2}$ intersects the upper lid at the point
$(u_1,u_2,1)$ and let $(u_1,u_2,1)\in P_A$. Then the projection of the line
$L_{t_1t_2}$ onto the upper lid is tangent to $P_A$ at point $(u_1,u_2,1)$.
\end{theorem}

\begin{proof}
Fix $t_1,t_2$, which fixes $a=a(t_2)$ and consider formulae~\eqref{line} that
give us $L_{t_1t_2}$. Then we can compute the slope of the projection of
$L_{t_1t_2}$ on plane $(X_1,X_2)$. In fact from~\eqref{line}
$$
\frac{\partial X_1}{\partial X_3} = -2t_1 a^2\,, \,\,\, \frac{\partial
X_2}{\partial X_3} = \frac{4t_1^2 a^3(1-2a)^2}{1-4a}\,.
$$
Therefore, the slope of this projection is the ratio of these quantities:
\begin{equation}
\label{slope1}
-\frac{2t_1a(1-2a)^2}{1-4a}\,.
\end{equation}

On the other hand the point of intersection of $L_{t_1t_2}$ wit the plane
$\{X_3=1\}$ is  (again see~\eqref{line})
$$
U=(u_1,u_2,1)\,,\,\,\text{where}\,\, u_1 = -t_1 a\,,\,\,u_2 = t_1^2\,a^2(1-2a)^2\,.
$$
Therefore, point $U$ lies on parabola $v_3=1, v_2 = \frac{(1-2a)^2}{1-4a}\cdot
v_1^2$, and, hence, the slope of this parabola at $U$ is
\begin{equation}
\label{slope2}
\frac{2(1-2a)^2}{1-4a}\cdot u_1\,.
\end{equation}
Now plug $u_1 = -t_1 a$ into~\eqref{slope2} to see that it is equal to the
formula for slope in \eqref{slope1}. We are done.

\end{proof}

\vspace{.1in}

Let us see what influence this theorem has on the choice of $\phi_n$.

Given $L=L_{t_1t_2}$ intersecting  the upper lid at $u=(u_1,u_2,1)$. To
construct $\phi$ corresponding to $u$ we make an infinitesimally small jumps
$u+\ve e$, $u-\ve e$ along $L$ (here $e$ is a unit vector parallel to $L$ and
with positive third coordinate). If we would know how to build the extremal
functions $\phi_{\pm}$ for $u\pm\ve e$ we would just put the one for $u+\ve e$
on the right half $I_+$ of $I$ and the one for $u-\ve e$ on the left half $I_-$
of $I$. The one for $u+\ve e$ is definitely unknown. But the one for $u-\ve e$
is ``known" in the sense that we can jump from it to $u-2\ve e=u-\ve e-\ve e$
and $u=u-\ve e+\ve e$. Then we can restart the procedure again. Thus, the
explanation for the self-similar structure of $\phi_n$ and for the second line
of~\eqref{phin}.

Notice that after several jumps along $L$ we will find ourselves at the
end-point $u-n\ve e$ of $L$ where $X_3=0, X_2=X_1^2$ (especially if we choose
dyadic number $\ve$). At this moment we know that extremal function is just
constant $X_1=:c_n\, u_1$. This is the explanation for the first line in the
definition of $\phi_n$.

To explain $d_n$ let us recall that we do not yet know how to build the
extremal function for the point $u+\ve e$. Moreover, we cannot do that. This
point is outside of $\Omega$. So here is heuristics. Instead of jumping into
$u+\ve e$ let us jump to point $p_{\ve}=$ projection of $u+\ve e$ onto the
upper lid. We know by Theorem~\ref{foliupper} that $p_{\ve}$ lies on the
tangent line to parabola $P_A$: $x_2 =A \,x_1^2$ at point $u$. So, in a sense,
we can think that $p_{\ve}$ lies on $P_A$ (because $\ve$ is very small). Or
rather we can choose $U=p_{\ve} + O(\ve^2)$ such that $U\in P_A$. Let
$U=(U_1,U_2,1)$.

Now we need to understand the following: suppose we know an almost extremal
function for $u$ (our future $\phi_n$). Do we know it for $U$ then? The answer
is yes, because if we know the extremal function for one point of parabola than
the extremal function for another point of the same parabola is acquired just
by multiplication of the first function on a suitable constant.

This is the explanation for the third line of \eqref{phin}.

\section{Step I for Carleson embedding theorems in $L^p$ with $p>1$}
\label{pICET}

Now, we briefly consider how to find the Bellman function for the Carleson
embedding operator acting on arbitrary $L^p$. We shell follow the same scheme
and we shall start with an evident definition of the Bellman function.

\begin{multline}
\Bel_{\max}(x_1,x_2,x_3;m,M):=\sup\Big\{\int_I
|\phi(t)|^p\,d\mu(t)\,+\!\!\sum_{J\in D(I)} |\,\av{\phi}J|^p \alpha_J\colon
\\
\av{\phi}I=x_1,\quad \av{|\phi|^p}I=x_2,\quad
\frac1{|I|}\Bigl(\mu(I)\,+\!\!\sum_{J\in D(I)}\alpha_J\Bigr)=x_3,
\\
m|J|\le\mu(J)\,+\!\!\sum_{\ell\in D(J)}\alpha_{\ell}\le M|J|\quad\forall J\in
D(I)\Big\}\,, \label{pcetbel}
\end{multline}
\begin{multline}
\Bel_{\min}(x_1,x_2,x_3;m,M):=\inf\Big\{\int_I
|\phi(t)|^p\,d\mu(t)\,+\!\!\sum_{J\in D(I)} |\,\av{\phi}J|^p \alpha_J\colon
\\
\av{\phi}I=x_1,\quad \av{|\phi|^p}I=x_2,\quad
\frac1{|I|}\Bigl(\mu(I)\,+\!\!\sum_{J\in D(I)}\alpha_J\Bigr)=x_3,
\\
m|J|\le\mu(J)\,+\!\!\sum_{\ell\in D(J)}\alpha_{\ell}\le M|J|\quad\forall J\in
D(I)\Big\}\,. \label{pcetbelmin}
\end{multline}

Functions $\Bel$ are defined in
\begin{equation}
\Omega:= \{x=(x_1,x_2,x_3)\colon\ x_1^p \le x_2 ,\ m\le x_3\le M\}\,.
\label{pcetdom}
\end{equation}

As before, the Bellman functions $\Bel$ do not depend on the interval $I$. The
difference is that now we consider non-negative test functions only.

From the definition of function $\Bel$ above we can immediately see that it
satisfies the inequality (just concavity or convexity)

\begin{align}
\label{pmiB} \Bel_{\max}(x) - \half \big(\Bel_{\max}(x^-) +
\Bel_{\max}(x^+)\big) &\ge 0\,,\quad\forall x^\pm\in\Omega\,,\quad x= \half
(x^+ + x^-)\,.
\\
\label{pmiBmin} \Bel_{\min}(x) - \half \big(\Bel_{\min}(x^-) +
\Bel_{\min}(x^+)\big) &\le 0\,,\quad\forall x^\pm\in\Omega\,,\quad x= \half
(x^+ + x^-)\,.
\end{align}
Note that $\Omega$ (see~\eqref{pcetdom}) is convex and for all $x^\pm\in\Omega$
we have $x\in\Omega$.

The boundary conditions also follow  from the definition:

\begin{equation}
\label{pBCB} \Bel(x_1, |x_1|^p, x_3) = |x_1|^p\, x_3\,.
\end{equation}

\begin{equation}
\label{pBCBUL} \frac{\partial\Bel}{\partial x_3} (x_1,x_2,M;m,M) = |x_1|^p\,.
\end{equation}

It is equally easy to see that homogeneity condition holds

\begin{equation}
\label{phmB} \Bel(tx_1, t^px_2, x_3) = t^p\,\Bel(x_1,x_2,x_3)\,,\quad t>0,.
\end{equation}

As before, we are looking for a function $B$ satisfying the following {\bf
degeneration}  condition:
\begin{equation}
\label{pdegdetB} \det (d^2 B) = \det \begin{pmatrix}
B_{x_1x_1} & B_{x_1x_2} & B_{x_1x_3}\\
B_{x_2x_1} & B_{x_2x_2} & B_{x_2x_3}\\
B_{x_3x_1} & B_{x_3x_2} & B_{x_3x_3}\end{pmatrix} = 0\,,\qquad \forall x\in
\Omega\,.
\end{equation}

\bigskip

\begin{lemma}
\label{claim}
 There is a simple algorithm to find the function
$B(x;m,M)$  that solves Monge--Amp\`ere equation~\eqref{pdegdetB} in the
domain~\eqref{pcetdom} with boundary condition~\eqref{pBCB}--\eqref{pBCBUL} and
homogeneity condition~\eqref{phmB}.
\end{lemma}

\bigskip

\begin{proof}
Again we look for a solution $B$ of the form
\begin{equation}
\label{pB} B(x_1, x_2, x_3) = t_1\cdot x_1 + t_2\cdot x_2 + t_3\cdot x_3 +
t_0\,,
\end{equation}
where $t_i=\frac{\partial B}{\partial x_i}$, $i=1,2,3$. The integral curves of
the vector field $\Theta$ such that $d^2B(x)\Theta(x)=0$ are again segments of
straight lines
\begin{equation}
\label{pinvlinesB} x_1 dt_1 + x_2 dt_2 + x_3 dt_3 + dt_0=0\,.
\end{equation}
Each line is given by fixing two free parameters, for example $t_1$ and $t_2$.
Then we can rewrite equation~\eqref{pinvlinesB} defining lines $L_{t_1, t_2}$
foliating our domain in the following form:
\begin{equation}
\label{plinesB}
\begin{cases}
x_1 + \threeone\cdot x_3 + \zeroone = 0
\\
x_2 + \threetwo\cdot x_3 + \zerotwo =0
\end{cases}
\end{equation}

Differentiating~\eqref{phmB} in $t$ and setting $t=1$ we get
\begin{equation}
x_1 t_1 + p x_2 t_2 = p B\,. \label{pshortB}
\end{equation}
Whence,
\begin{equation}
\label{pplaneB} (p-1)t_1 x_1 + p t_3 x_3 + p t_0 =0\,.
\end{equation}

If we subtract this equation from the first equation of~\eqref{plinesB}
multiplied by $(p-1)t_1$, then we get
$$
\big(pt_3-(p-1)t_1\,\threeone\big)x_3+\big(pt_0-(p-1)t_1\,\zeroone\big)=0
$$
for every $x_3$. If we introduce the dual exponent $q=\frac{p}{p-1}$, then the
latter equation can be rewritten as follows
$$
\big(qt_3-t_1\,\threeone\big)x_3+\big(qt_0-t_1\,\zeroone\big)=0\,.
$$
Hence
\begin{equation}
\begin{cases}
\displaystyle q t_3-t_1\,\threeone=0\\
\displaystyle q t_0-t_1\,\zeroone=0\,.
\end{cases}
\end{equation}

From these ``PDE" we easily write down
\begin{equation}
\label{pAD}
\begin{cases}
t_3 = A(t_2) |t_1|^q\\
t_0 = D(t_2) |t_1|^q\,.
\end{cases}
\end{equation}
Then the equations of the extremal lines $L_{t_1,t_2}$ can be rewritten in the
form
\begin{equation}
\label{plinesAD}
\begin{cases}
x_1 + q |t_1|^{q-2}t_1 A x_3 + q |t_1|^{q-2}t_1 D=0
\\
x_2 + |t_1|^q A' x_3 + |t_1|^q D' =0
\end{cases}
\end{equation}

As before, we assume that our extremal line $L_{t_1,t_2}$ intersect the
boundary $\partial\Omega$ in a point $\zeta=\zeta(t_1,t_2)$ on the ``upper
lid'' $x_3=M$ and in a point $\xi=\xi(t_1,t_2)$ on the ``side'' $x_2=|x_1|^p$.
Then we have two pairs of equations~\eqref{plinesAD} asserting that our points
are on the line $L_{t_1,t_2}$
\begin{align}
\zeta_1 + q |t_1|^{q-2}t_1 A M + q |t_1|^{q-1}t_1 D=&0\label{pzeta1}
\\
\zeta_2 + |t_1|^q A' M + |t_1|^q D'=&0\label{pzeta2}
\\
\xi_1 + q |t_1|^{q-2}t_1 A \xi_3 + q |t_1|^{q-2}t_1 D=&0\label{pxi11}
\\
|\xi_1|^p + |t_1|^q A' \xi_3 + |t_1|^q D'=&0\label{pxi12}
\end{align}
and two boundary conditions~\eqref{pBCB} and~\eqref{pBCBUL}
\begin{align}
\frac1p\xi_1 t_1 + |\xi_1|^p t_2=&|\xi_1|^p\xi_3\label{pxi13}
\\
A |t_1|^q=&|\zeta_1|^p\,.\label{pA}
\end{align}

If we take $\zeta_1$ from~\eqref{pzeta1} and plug into~\eqref{pA} we get
\begin{equation}
A=q^p|AM+D|^p\,.\label{pA1}
\end{equation}
So, introducing a new function of $t_2$
\begin{equation}
a=q(AM+D)\label{pa}
\end{equation}
we can express both functions $A$ and $D$ in terms of $a$: \eqref{pA1} yields
\begin{equation}
A=|a|^p\label{pA2}
\end{equation}
and directly from the definition of $a$~\eqref{pa} we get
\begin{equation}
D=\frac1q a-M|a|^p\,.\label{pD}
\end{equation}

In result equations~\eqref{pxi11}, \eqref{pxi12}, and~\eqref{pxi13} can be
rewritten as follows
\begin{align}
\xi_1 + |t_1|^{q-2}t_1 a \big(1-qa|a|^{p-2}(M-\xi_3)\big)=&0\label{pxi21}
\\
|\xi_1|^p + \frac pq|t_1|^q a'\big(\frac1p-qa|a|^{p-2}(M-\xi_3)\big)=&0
\label{pxi22}
\\
\frac1p\xi_1 t_1 -|\xi_1|^p(\xi_3-t_2)=&0\,.\label{pxi23}
\end{align}

As before, we introduce the following new function
\begin{equation}
\eta=1-qa|a|^{p-2}(M-\xi_3)\,,\label{peta}
\end{equation}
and then equations~\eqref{pxi21} and~\eqref{pxi22} turn into
\begin{align}
\xi_1 + |t_1|^{q-2}t_1 a \eta=&0\label{pxi31}
\\
|\xi_1|^p + \frac pq |t_1|^q a'\big(\eta-\frac1q\big)=&0\label{pxi32}\,,
\end{align}
whence
\begin{equation}
a'\big(\eta-\frac1q\big)+\frac qp |a|^p|\eta|^p=0\,.\label{peta1}
\end{equation}

Since $a$ does not depend on $t_1$, the function $\eta$ (as a solution of this
equation) also depends only on $t_2$.

On the one side, from~\eqref{peta} we have
\begin{equation}
\xi_3=M-\frac{1-\eta}{qa|a|^{p-2}}\,,\label{pxi3a}
\end{equation}
on the other side~\eqref{pxi23} together with~\eqref{pxi31} yield
\begin{equation}
\xi_3=t_2+\frac{\xi_1t_1}{p|\xi_1|^p}=t_2-\frac{a\eta}{p|a|^p|\eta|^p}\,.
\label{pxi3b}
\end{equation}
So, excluding $x_3$ we get
$$
t_2-\frac{a\eta}{p|a|^p|\eta|^p}=M-\frac{1-\eta}{qa|a|^{p-2}}\,.
$$
Differentiating in $t_2$ the latter equation we obtain
$$
1+\frac{a'\eta+\eta'a}{q|a|^p|\eta|^p}=\frac{\eta'}{qa|a|^{p-2}}+\frac{p-1}q
\cdot\frac{(1-\eta)a'}{|a|^p}
$$
and then we multiply this equality by the common denominator and replace it
using~\eqref{peta1} by $-pa'(\eta-\frac1q)$:
$$
-pa'(\eta-\frac1q)+a'\eta+\eta'a=\eta'a|\eta|^p+(p-1)(1-\eta)a'|\eta|^p\,,
$$
or
\begin{equation}
\label{paeta} (p-1)(|\eta|^p-1)(\eta-1) a' = a(|\eta|^p-1)\eta'\,.
\end{equation}

Thus, we come to the following equation
\begin{equation}
\label{pab1} (p-1)(\eta-1) a' =  a\eta'\,,
\end{equation}
which implies $\eta=1+C|a|^p$. Then~\eqref{pxi3a} yields
\begin{equation}
\xi_3(t_2)=M+\frac{C|a|^{p-1}}{qa|a|^{p-2}}= M+\frac{C}q\sign{a} =: c\,.
\label{px3}
\end{equation}
The latter expression is a constant, because $\sign{a}$ can take the only value
due to $\xi_3\le M$. As before, the natural choice of the constant $c$ is
$c=m$.

Recall that we have the following expression for $B$ (see~\eqref{pshortB})
\begin{equation}
\label{pshortB1} B = \frac1p t_1x_1+t_2x_2\,.
\end{equation}
Thus, to find an expression for $B$ we need to find $t_1$ and $t_2$ as
functions of a point $x$ running over our domain $\Omega$.

From~\eqref{peta} we have
$$
\eta=1-qa|a|^{p-2}(M-m)\,,
$$
hence~\eqref{pxi3b} yields
\begin{equation}
\label{pt2} t_2 =
m+\frac{a\eta}{p|a\eta|^p}=m+\frac{a-q(M-m)|a|^p}{p\big|a-q(M-m)|a|^p\big|^p}\,.
\end{equation}
This gives us the desired $a$ as a function of $t_2$, but it is clear that $a$
is more convenient parameter than $t_2$. We shall express all other functions
of $t_2$ in terms of $a$ and look for $a$ as a function of $x\in\Omega$. To
this aim we return to the equations of the extremal lines~\eqref{plinesAD}
rewriting them in terms of $a$. From~\eqref{pt2} we have
$$
\frac{dt_2}{da}=-\frac1{q|a\eta|^p}\cdot\frac{d}{da}(a\eta)=
-\frac{1-pq(M-m)|a|^{p-2}a}{q\big|a-q(M-m)|a|^p\big|^p}
$$
therefore
$$
a'=-\frac{qa\big|a-q(M-m)|a|^p\big|^p}{a-pq(M-m)|a|^p}
$$
and
$$
A=|a|^p\,,A'=p|a|^{p-2}aa'=
-pq|a|^p\,\frac{\big|a-q(M-m)|a|^p\big|^p}{a-pq(M-m)|a|^p}\,,
$$
$$
D=\frac1q a-M|a|^p\,,D'=\frac1q\big(1-pqM|a|^{p-2}a\big)a'=
$$
$$
-\frac{\big(a-pqM|a|^p\big)\big|a-q(M-m)|a|^p\big|^p}{a-pq(M-m)|a|^p}\,.
$$

Now equation~\eqref{plinesAD} of the line $L_{t_1, t_2}$ becomes
\begin{equation}
\label{pline}
\begin{cases}
\displaystyle x_1 = -|t_1|^{q-2}t_1\,\big(a-q(M-x_3)|a|^p\big)\,,
\rule[-15pt]{0pt}{15pt}
\\
\displaystyle x_2 =
|t_1|^q\frac{\big|a-q(M-m)|a|^p\big|^p\big(a-pq(M-x_3)|a|^p\big)}{a-pq(M-m)|a|^p}
\,.
\end{cases}
\end{equation}
To get a more symmetrical expression, as before, we take a more geometrical
parameter $\xi_1$ instead of $t_1$, where $\xi_1$ is the first coordinate of
$\xi=(\xi_1,|\xi_1|^p,m)$, the point where the side surface $\{x_2=|x_1|^p\}$
is intersected by our extremal line. In term of $a$ and $\xi_1$
equations~\eqref{pline} turns into
\begin{equation}
\label{pline1}
\begin{cases}
\displaystyle x_1 = \frac{a-q(M-x_3)|a|^p}{a-q(M-m)|a|^p}\,\xi_1\,,
\rule[-20pt]{0pt}{20pt}
\\
\displaystyle x_2 = \frac{a-pq(M-x_3)|a|^p}{a-pq(M-m)|a|^p}\,|\xi_1|^p\,.
\end{cases}
\end{equation}

These equations immediately supply us with an expression for $a$, namely and
$a=a(x)=$ is a root of the following equation
\begin{equation}
\label{pcubic} s:= \frac{|x_1|^p}{x_2} = \left|
\frac{a-q(M-x_3)|a|^p}{a-q(M-m)|a|^p} \right|^p
\frac{a-pq(M-m)|a|^p}{a-pq(M-x_3)|a|^p}\,.
\end{equation}

To determine which of possible roots of the equation~\eqref{pcubic} gives us
the desired value of $a$, we investigate the above function $s=s(a)$ defined
by~\eqref{pcubic} as a function of the parameter $a$, all other parameters
assuming to be fixed.

First of all we note that the extremal line~\eqref{pline1} intersects the plane
$\{x_3=M\}$ at the point
\begin{equation}
\zeta=(\zeta_1,\zeta_2,M)\,,\quad \zeta_1= \frac{a\xi_1}{a-pq(M-x_3)|a|^p}\,,
\quad \zeta_2 = \frac{a|\xi_1|^p}{a-pq(M-x_3)|a|^p}\,. \label{pzeta}
\end{equation}
Since $\zeta_2\ge0$, we have the first restriction for $a$:
\begin{equation}
\label{pdoma} a<\Big(\frac1{pq(M-m)}\Big)^{q-1}\,.
\end{equation}
The behavior of $s(a)$ on the semi-axis~\eqref{pdoma} is the same for all
values of other parameters: on the negative half-line it monotonously increases
from $\big(\frac{M-x_3}{M-m}\big)^{p-1}$ till $1$, and on the interval
$[0,\big(\frac1{pq(M-m)}\big)^{q-1}]$ it monotonously decreases from $1$ to
$0$. Therefore, for a given $s$, $0\le s\le1$, we have one or two solutions
of~\eqref{pcubic} satisfying~\eqref{pdoma}: for all $s$ there exists a positive
solution $a=a^+$ and for $s>\big(\frac{M-x_3}{M-m}\big)^{p-1}$ there exists a
negative solution $a=a^-$ as well. We already know that the solution $a=a^+$
corresponds to the function $\Bel_{\max}$, and the solution $a=a^-$ corresponds
to the function $\Bel_{\min}$.

To write down the expression for $B$ we use~\eqref{pshortB}, then we express
$x_1$ using the first equation of~\eqref{pline}, and by means of the second
equation of~\eqref{pline} we replace the expression of $|t_1|^q$. For $t_2$ we
use~\eqref{pt2}. In result we obtain
\begin{multline*}
B=\frac1p x_1t_1+x_2t_2=-\frac1p|t_1|^q\big(a-q(M-x_3)|a|^p\big)+x_2t_2=
\rule[-20pt]{0pt}{20pt}
\\
-\frac{\big(a-q(M-x_3)|a|^p\big)\big(a-pq(M-m)|a|^p\big)}
{p\big|a-q(M-m)|a|^p\big|^p\big(a-pq(M-x_3)|a|^p\big)}x_2+
x_2\Big(m+\frac{a-q(M-m)|a|^p}{p\big|a-q(M-m)|a|^p\big|^p}\Big)=
\rule[-25pt]{0pt}{25pt}
\\
\frac{a|a|^p(x_3-m)x_2}{\big|a-q(M-m)|a|^p\big|^p\big(a-pq(M-x_3)|a|^p\big)}
+mx_2=\rule[-25pt]{0pt}{25pt}
\\
\frac{(x_3-m)x_2}{\big|1-q(M-m)a|a|^{p-2}\big|^p\big(1-pq(M-x_3)a|a|^{p-2}\big)}
+mx_2\,.
\end{multline*}
As it was already mentioned, here we choose the solution $a=a^+$
of~\eqref{pcubic} for $B_{\max}$ and $a=a^-$ for $B_{\min}$. 


\begin{theorem}
\label{pbcmin} Consider
\begin{equation}
\label{peqcet1min} B(x_1,x_2,x_3;m,M)=
\frac{(x_3-m)x_2}{\bigl[1-2a(M-m)\bigr]^2\bigl[1-4a(M-x_3)\bigr]}+mx_2 \,.
\end{equation}
We shall denote this function by $B_{\max}$ if $a=a^+(x)$ is the unique
positive solution of the equation
$$
\frac{|x_1|^p}{x_2} = \left| \frac{a-q(M-x_3)|a|^p}{a-q(M-m)|a|^p} \right|^p
\frac{a-pq(M-m)|a|^p}{a-pq(M-x_3)|a|^p}
$$
from the interval $[0,\big(\frac1{pq(M-m)}\big)^{q-1}]$, and this function will
be denoted by $B_{\min}$ if $a=a^-(x)$ is the unique negative solution of the
same equation. The domain of $B_{\max}$ is
$$
\Dom(B_{\max})=\{x\colon x_2\ge|x_1|^p,\,m\le x_3\le M\}\,,
$$
the domain of $B_{\min}$ is
$$
\Dom(B_{\min})=\{x\colon |x_1|^p\le x_2
\le\big(\frac{M-m}{M-x_3}\big)^{p-1}|x_1|^p,\,m\le x_3\le M\}\,.
$$
Then
$$
\Bel_{max}(x;m,M) =
\begin{cases}
B_{\max}(x;m,M)\quad& m<x_3\le M \,,\rule[-15pt]{0pt}{15pt}
\\
\qquad mx_2&\quad x_3=m\,.
\end{cases}
$$
and
$$
\Bel_{min}(x;m,M) =
\begin{cases}
B_{\min}(x;m,M)\quad&\displaystyle M-(M-m)\frac{|x_1|^q}{x_2^{q-1}}\le x_3\le
M\,, \rule[-15pt]{0pt}{15pt}
\\
\qquad mx_2&\displaystyle m\le x_3\le M-(M-m)\frac{|x_1|^q}{x_2^{q-1}}\,.
\end{cases}
$$
\end{theorem}

\end{proof}

\bigskip

The proof of this theorem follows almost verbatim the lines of the proof of our Theorems \ref{bc}, \ref{bcmin} above. So, we are not going to prove it here. Instead, we wish to show that our results encompass a slightly more general situation than it could have been thought.

\bigskip

\section{Trees.}
\label{trees}

As in \cite{M} we let $(X,\mu_0)$ be a nonatomic probability space. Two measurable subsets $A,B$ of $X$ will be called almost disjoint if $\mu_0(A\cap B) =0$. Then we give the following 

\medskip

\noindent{\bf Definition.} A set $\mathcal{T}$ is called  a tree if the following conditions are satisfied 
\begin{itemize}
\item $X\in \mathcal{T}$ and for every $I\in \mathcal{T}$ we have $\mu_0(I)>0$.
\item For every $I \in \mathcal{T}$ there corresponds a  finite or countable subset $C(I)\subset \mathcal{T}$ containing at least two elements such that:

a) the elements of $C(I)$ are pairwise almost disjoint,

b) $I=\cup_{J\in C(I)} J$\,.
\item $\mathcal{T}= \cup_{m\ge 0} \mathcal{T}_m$, where $\mathcal{T}_0=\{X\}$,
$\mathcal{T}_{m+1} = \cup_{I\in \mathcal{T}_m} C(I)$.
\item We have $\lim_{m\rightarrow\infty} \sup_{I\in \mathcal{T}_m}\mu_0(I)=0$.
\end{itemize}

\medskip

For any tree $\mathcal{T}$ we define its exceptional set $E= E(\mathcal{T})$ as $$
E= \cup_{I\in \mathcal{T}}\cup_{J_1,J_2 \in C(I), J_1\neq J_2} J_1\cap J_2\,.
$$
It is clear that $E$ has $\mu_0$ measure zero.

An easy induction shows that each $\mathcal{T}_m$ consists of pairwise almost disjoint sets whose union is $X$. Moreover, if $x\in X\setminus E$, then for each $m$ there exists exactly one $I_m$ in $\mathcal{T}_m$ containing $x$. For every $m>0$ there exists a $J\in \mathcal{T}_{m-1}$ such that $I_m\in C(J)$.
Then of course $J=I_{m-1}$. Hence the set $A(x) = \{ I\in \mathcal{T}: x\in I\}$ forms a chain of nested elements of $\mathcal{T}_m, m=0,1,...$. Inclusions are all strict. From this remark it clear that if $I,J\in \mathcal{T}$ and $I\cap J\cap(X\setminus E)$ is nonempty then $I\subset J$ or $J\subset I$. In particular, for any $I,J\in \mathcal{T}$  we have either $\mu_0(I\cap J) =0$ or 
one of them is contained in the other. We denote by $\mathcal{T}(J)$ all elements of $\mathcal{T}$ that are subsets of $J$.
Also for any $I\in \mathcal{T}$ put
$$
\av{\phi}I := \frac1{\mu_0(I)}\int_I \phi\,d\mu_0\,.
$$

Given any tree $\mathcal{T}$ consider the collection $\alpha=\{\alpha_I\}_{I\in \mathcal{T}}$
of non-negative numbers. Let $\mu$ be another measure on $X$.
We can introduce the tree Bellman function:

\begin{multline}
\Bel_{CET}^{\mathcal{T}}(x_1,x_2,x_3;m,M):=\sup\Big\{\int_X
|\phi(t)|^2\,d\mu(t)\,+\!\!\sum_{J\in \mathcal{T}}
 |\,\av{\phi}J|^2 \alpha_J\colon
\\
\av{\phi}X=x_1,\quad \av{\phi^2}X=x_2,\quad
\Bigl(\mu(X)\,+\!\!\sum_{J\in \mathcal{T}}\alpha_J\Bigr)=x_3,
\\
m\mu_0(J)\le\mu(J)\,+\!\!\sum_{\ell\in \mathcal{T}(J)}\alpha_{\ell}\le M\mu_0(J)\quad\forall J\in
\mathcal{T}\Big\}\,. \label{cetbel}
\end{multline}

Independently of the tree the same Theorem \ref{bc} holds:

\begin{theorem}
\label{bctree} Consider
\begin{equation}
\label{eqcet1tree} B(x_1,x_2,x_3;m,M)=
\frac{(x_3-m)x_2}{\bigl[1-2a(M-m)\bigr]^2\bigl[1-4a(M-x_3)\bigr]}+mx_2 \,,
\end{equation}
where $a= a(x)$ is the unique solution of the cubic equation
\begin{equation}
\label{cubictree}
\frac{x_1^2}{x_2} = \left[  \frac{1-2a(M-x_3)}{1-2a(M-m)} \right]^2
\frac{1-4a(M-m)}{1-4a(M-x_3)}
\end{equation}
on the interval $[0,\frac1{4(M-m)}]$. Then
$$
\Bel_{CET}^{\mathcal{T}}(x;m,M) =
\begin{cases}
B(x;m,M)\quad& x_1^2 \le x_2,\quad m<x_3\le M
\\
\qquad mx_2 &x_1^2 \le x_2,\quad x_3=m.
\end{cases}
$$
\end{theorem}

We can introduce the tree Bellman function also for $p\neq 2, 1<p<\infty$. Again the result for the trees  will not depend on the tree, it coincides with dyadic result. The proofs of $\Bel_{CET}^{\mathcal{T}}(x;m,M) \le B(x;m,M)$ is exactly the same as before, the proof of $\Bel_{CET}^{\mathcal{T}}(x;m,M) \ge B(x;m,M)$ is not too difficult either if one uses the following lemma from \cite{M}.

\begin{lemma}
\label{melasLemma}
For every $I\in \mathcal{T}$ and every $\alpha$ such that $0<\alpha <1$ there exists a subfamily $F(I)\subset \mathcal{T}$ consisiting of pairwise almost disjoint subsets of $I$ such that
$$
\mu_0(\cup_{J\in F(I)} J ) = (1-\alpha) \mu_0(I)\,.
$$
\end{lemma}


    \end{document}